\RequirePackage{iftex}
\ifpTeX
  \PassOptionsToPackage{dvipdfmx}{graphicx}
  \PassOptionsToPackage{dvipdfmx}{xcolor}
  \PassOptionsToPackage{dvipdfmx}{hyperref}
  
\fi
\documentclass[a4paper,11pt]{amsart}

\usepackage[english]{babel}
\usepackage[utf8]{inputenc}
\usepackage[T1]{fontenc}

\usepackage{amsmath,amssymb,amsthm,amscd}
\usepackage{mathtools}
\usepackage{mathrsfs}
\usepackage{calligra}

\usepackage{enumitem}
\usepackage{comment}
\usepackage{float}

\usepackage{tikz}
\usetikzlibrary{calc,arrows,matrix,shapes.geometric}

\usepackage{tikz-cd}

\usepackage[all]{xy}

\usepackage{hyperref}
\hypersetup{
  colorlinks=true,
  allcolors=cyan
}

\usepackage{amsthm}
\usepackage{geometry}
\newtheorem{theorem}{Theorem}[section]
\newtheorem{proposition}[theorem]{Proposition}
\newtheorem{claim}[theorem]{Claim}
\newtheorem{lemma}[theorem]{Lemma}
\newtheorem{corollary}[theorem]{Corollary}

\theoremstyle{definition}
\newtheorem{definition}[theorem]{Definition}

\newtheorem{example}[theorem]{Example}
\newtheorem{remark}[theorem]{Remark}

\newtheorem{setup}[theorem]{Setup}

\newcommand{\Hom}[0]{\operatorname{Hom}}

\newcommand{\End}[0]{\operatorname{End}}

\newcommand{\Exc}{\mathrm{Exc}}

\newcommand{\dbar}{\overline{\partial}}

\def \CC {\mathbb C} %
\def \PP {\mathbb P}  %
\DeclareMathOperator{\id}{id} %

\let\Gamma\varGamma
\let\Delta\varDelta
\let\Theta\varTheta
\let\Lambda\varLambda
\let\Xi\varXi
\let\Pi\varPi
\let\Sigma\varSigma
\let\Upsilon\varUpsilon
\let\Phi\varPhi
\let\Psi\varPsi
\let\Omega\varOmega

\newcommand{\C}{\mathbb{C}}
\newcommand{\R}{\mathbb{R}}
\newcommand{\B}{\mathbb{B}}

\newcommand{\cO}{\mathcal{O}}

\newcommand{\cE}{\mathcal{E}}
\newcommand{\ddc}{\sqrt{-1}\partial\dbar}
\newcommand{\supp}{\operatorname{supp}}
\newcommand{\dist}{\operatorname{dist}}

\begin{document}
\title{Atiyah classes and ellipticity of Oka manifolds}

\author[Y. Kusakabe]{Yuta Kusakabe}
\address{Faculty of Mathematics, Kyushu University, Fukuoka 819-0395, Japan}
\email{kusakabe@math.kyushu-u.ac.jp}

\author[S. Matsumura]{Shin-ichi Matsumura}
\address{Mathematical Institute 
$\&$ Division for the Establishment of Frontier Science of Organization for Advanced Studies, 
Tohoku University, 
6-3, Aramaki Aza-Aoba, Aoba-ku, Sendai 980-8578, Japan}
\email{mshinichi-math@tohoku.ac.jp}
\email{mshinichi0@gmail.com}

\renewcommand{\subjclassname}{%
\textup{2020} Mathematics Subject Classification}
\subjclass[2020]{Primary 32Q56; Secondary 32L05, 32L10, 32Q15, 32W05, 14J28.}
\keywords{Oka manifold, elliptic manifold, dominating spray, Gromov's conjecture, weakly pseudoconvex manifold, Atiyah class, blow-up, Kummer surface}

\begin{abstract}
We prove that every weakly pseudoconvex Oka manifold admitting
a positive line bundle is elliptic in the sense of Gromov, 
thereby proving Gromov's ellipticity conjecture.

The proof establishes a cohomological construction of local
sprays. Given a bundle morphism $\varphi\colon E\to T_X$, we
introduce quadratic $\varphi$-vector fields on $E$ whose flows
yield local sprays with fibre derivative $\varphi$. Their
existence is characterized by the vanishing of the symmetrized
$\varphi$-Atiyah class of $E$. This gives a local dominating
spray on every weakly pseudoconvex manifold admitting a
positive line bundle, without any Oka assumption.

The positivity hypothesis is essential even for local
existence: blow-ups at a single point of complex tori of algebraic
dimension zero, and Kummer
surfaces of algebraic dimension zero, admit no local dominating
spray. The torus examples give compact K\"ahler Oka manifolds
that are not elliptic and show that ellipticity is not preserved
under blowing up, and that it is neither open nor closed in
holomorphic families of compact Oka manifolds. Combined with the
work of Xie and Zhao, the Kummer examples yield Oka K3 surfaces
that are not elliptic.

Under the additional Oka assumption, we globalize the local
dominating sprays constructed above. The proof combines
Oka approximation with a gluing argument based on weighted
$L^2$ estimates for the $\bar\partial$-equation.
\end{abstract}

\maketitle

\vspace*{-2\baselineskip}
\tableofcontents

\newpage

\section{Introduction}\label{sec-intro}

We investigate the relation between the Oka property and ellipticity
through the existence and globalization of local dominating
sprays. Our main result is the following.

\begin{theorem}[Gromov's conjecture for weakly pseudoconvex manifolds]\label{thm-main}
Every weakly pseudoconvex Oka manifold admitting a positive line bundle is elliptic in the sense of Gromov.
\end{theorem}

A complex manifold $X$ is called \emph{Oka} if every continuous
map from a Stein manifold to $X$ satisfies the homotopy principle;
that is, it can be deformed into a holomorphic map, with
approximation and interpolation
(see \cite[Chap.~5]{Forstneric2017book} for the precise definition,
and \cite{Forstneric2006} and \cite[Sect.~5.4]{Forstneric2017book}
for equivalent formulations).
A complex manifold $X$ is called \emph{elliptic} if a vector
bundle $p \colon  E \to X$ admits a dominating spray; that is, a
holomorphic map $s\colon E\to X$ from the total space of $E$
such that $s(0_x)=x$ and $(ds)_{0_x}(E_x)=T_{X,x}$ for every $x\in X$.

Gromov proved that every elliptic manifold is Oka \cite{Gromov1989JAMS}
(see \cite[Chap.~5]{Forstneric2017book} for a proof and
\cite[Sect.~6.4]{Forstneric2017book} for examples).
He stated, without proof, that the converse holds for holomorphically
convex manifolds admitting a positive line bundle
\cite[3.2.A$''$]{Gromov1989JAMS}. We refer to this as
\textit{Gromov's conjecture}.
He also observed that the converse holds for Stein manifolds
\cite[3.2.A]{Gromov1989JAMS}.
The converse fails without additional assumptions: the first author
found Oka manifolds that are not elliptic \cite{Kusakabe2020}.
For instance, if $n\ge2$, the complement $\mathbb C^n\setminus K$ of a compact
polynomially convex set $K$ is Oka \cite[Theorem~1.6]{Kusakabe2024},
whereas it is not elliptic if $n\ge3$ and $K$ has nonempty interior
\cite{AndristShcherbinaWold2016}.
All these examples are noncompact.
Recently, Forstneri\v{c} and L\'arusson settled the projective
case \cite{Forstneric2026MRL}.
Theorem~\ref{thm-main} unifies and extends these results,
proving Gromov's conjecture for weakly pseudoconvex manifolds
admitting a positive line bundle.
The class of weakly pseudoconvex manifolds includes
Stein, projective, and holomorphically convex manifolds
and arises naturally in several complex variables.

The positivity hypothesis cannot be omitted, even among compact
K\"ahler Oka manifolds. Our second main result shows that this obstruction
already occurs at the level of local dominating sprays.

\begin{theorem}[Compact K\"ahler manifolds without local dominating sprays;
see Examples~\ref{ex:torus-blowup-no-local-spray} and~\ref{ex:kummer-no-local-spray}]\label{thm-example}
\ \ 
\begin{itemize}
\item[$(1)$] Let $\mathbb T$ be a complex torus of algebraic dimension $0$.
Then the blow-up $\pi\colon X:=\operatorname{Bl}_{t_0}(\mathbb T)\to\mathbb T$ at a point $t_0 \in \mathbb T$ admits no local
dominating spray.

\item[$(2)$] Every Kummer surface $X$ of algebraic dimension $0$
admits no local dominating spray.

\end{itemize}
Here a \textit{local dominating spray} is 
a dominating spray $s \colon U \to X$ 
defined only on an open neighbourhood $U$ of the zero section. 
\end{theorem}

The manifold $X$ in the first assertion is Oka by
\cite[Corollary~6.4.13]{Forstneric2017book}
(see also Buzzard's theorem on tame sets \cite{Buzzard2003}).
Thus, it gives the first
example of a compact K\"ahler Oka manifold that is not elliptic,
answering a question of Forstneri\v{c} and L\'arusson
\cite{Forstneric2026MRL}.
The same example shows that ellipticity need not be preserved
under blowing up, even among compact K\"ahler manifolds.
Moreover, varying the torus in holomorphic families shows that
ellipticity is neither an open nor a closed condition in holomorphic
families of compact Oka manifolds
(Remark~\ref{rem:torus-blowup-families}); recall that the Oka property
itself is not closed in families \cite{ForstnericLarusson2014IMRN}.

Combined with the work of Xie and Zhao \cite{XieZhao2026},
the second assertion yields Oka K3 surfaces that are not
elliptic. It also gives another instance of the failure of
openness: there is a holomorphic family of Kummer surfaces whose
central fibre is the elliptic Fermat quartic, whereas all fibres
outside a countable set of parameters are not elliptic.

\medskip

We divide the construction of dominating sprays in the proof of
Theorem~\ref{thm-main} into two problems: 

\begin{enumerate}
\item[$(1)$]
\emph{Local existence.}
When does a complex manifold $X$, without any Oka assumption,
admit a vector bundle $E\to X$ carrying a local dominating spray?

\item[$(2)$]
\emph{Globalization.}
Given a local dominating spray on a weakly pseudoconvex
Oka manifold, under what additional conditions can we
construct a global dominating spray?
\end{enumerate}

A local dominating spray $s\colon U\to X$ on a vector bundle
$E\to X$ induces a surjective bundle morphism $E\to T_X$
through its fibre derivative. This leads to the following question: 

\begin{enumerate}
\item[$\bullet$]
Given a surjective bundle morphism
$\varphi\colon E\to T_X$, when can we integrate $\varphi$
to a local spray on $E$ whose fibre derivative is $\varphi$?
\end{enumerate}

We develop a cohomological approach to the local existence problem.
For this purpose, we introduce \emph{quadratic $\varphi$-vector fields}
on the total space of $E$ (see Definition~\ref{def:q-spray} and
Remark~\ref{rem:APS}). Their flows integrate $\varphi$ to a local
spray, called an \emph{exponential $\varphi$-spray}
(see Proposition~\ref{prop-integral}).

The existence of these vector fields is characterized by the
following cohomological criterion using the Atiyah class \cite{Atiyah1957}.

\begin{theorem}[Cohomological criterion for quadratic $\varphi$-vector fields, Theorem~\ref{thm:criterion}]
The following conditions are equivalent:
\begin{enumerate}
\item[$(1)$] The bundle $E$ admits a quadratic $\varphi$-vector field.
\item[$(2)$] The symmetrized $\varphi$-Atiyah class
$\alpha_\varphi(E)$ vanishes.
\end{enumerate}
\end{theorem}

Here $\alpha_\varphi(E)$ is defined in terms of the Atiyah class
$a(E)$ of $E$ (see Definition~\ref{def:atiyah}).
The theorem identifies the obstruction to constructing
quadratic $\varphi$-vector fields and gives a sufficient
condition for integrating a prescribed bundle morphism into a
local spray. When $\varphi$ is surjective, the resulting spray is dominating.
Applying this criterion, we construct a local dominating spray on every weakly
pseudoconvex manifold admitting a positive line bundle
(Theorem~\ref{thm:weakly}). 
This construction does not require the Oka property.

\smallskip 

The Oka property enters in the globalization step, expressed
by the following theorem.

\begin{theorem}\label{thm:local-to-global-spray}
Let $X$ be a weakly pseudoconvex Oka manifold, and let $B\to X$ be a
positive line bundle such that $K_X^{-1}\otimes B^{\otimes k}$ admits a
smooth Hermitian metric with semipositive curvature for some integer $k\ge0$.
If $E=(B^\vee)^{\oplus N}$ admits, for some $N\ge1$, a local dominating
spray $s$, then $E$ admits a global dominating spray that
agrees with $s$ to order $2$ along the zero section.
In particular, $X$ is elliptic.
\end{theorem}

Here two holomorphic maps \emph{agree to order $r$} along a
submanifold $A$ if, in local coordinates on the target, all
their partial derivatives of order $<r$ agree at the points
of $A$.
Theorem~\ref{thm:local-to-global-spray} follows from a jet
extension theorem for maps from the total space of a negative
vector bundle into an Oka manifold, with interpolation along
the zero section (see Theorem~\ref{thm:oka}).

In the projective case, Forstneri\v{c} and L\'arusson used
Stopar's Oka principle on $1$-convex domains
\cite[Theorem~3.3]{Stopar2013IJM} 
(cf.~\cite{Prezelj2010}).
In our setting, the zero section may be noncompact, so $E$
need not be $1$-convex and these results do not apply in general.
The domination off the zero section required in
Theorem~\ref{thm:oka} is provided by
Lemma~\ref{lem:fibresubmersive}, the analogue of
\cite[Lemma~2.2]{Forstneric2026MRL}.
The proof of Theorem~\ref{thm:oka} follows the local-to-global
scheme of the Oka principle for Stein sources \cite[Chap.~5]{Forstneric2017book},
with a gluing argument based on weighted $L^2$ estimates for
the $\bar\partial$-equation that preserves the prescribed jets
along the zero section.
Here we treat the jet extension theorem only in the form
needed to globalize sprays.
The method also extends to maps from general weakly pseudoconvex
manifolds into Oka manifolds; we will address this elsewhere.

\emph{Organization.}
Section~\ref{sec:local-spray} treats local sprays:
Section~\ref{subsec:criterion} establishes the cohomological
criterion, and Section~\ref{subsec:applications} proves
the nonexistence results above.
Section~\ref{sec:globalization} is devoted to globalization:
Section~\ref{subsec:heftung} proves a splitting lemma on the
weakly pseudoconvex total space of a negative bundle, and
Section~\ref{sec:total} proves Theorem~\ref{thm:oka}.
Section~\ref{sec:proofmain} proves Theorem~\ref{thm:weakly} and
Theorems~\ref{thm-main} and~\ref{thm:local-to-global-spray}.

\section{Existence of local dominating sprays}
\label{sec:local-spray}

\subsection{Quadratic \texorpdfstring{$\varphi$}{phi}-vector fields and a cohomological integration criterion}
\label{subsec:criterion}

In this subsection, for a given bundle morphism $\varphi\colon E\to T_X$, 
we introduce quadratic $\varphi$-vector fields
and show, by considering their flows, that $\varphi\colon E\to T_X$ can be integrated to local sprays,
called exponential $\varphi$-sprays
(see Proposition~\ref{prop-integral}).
We also give a cohomological criterion
for the existence of quadratic $\varphi$-vector fields (see Theorem~\ref{thm:criterion}).

Throughout this subsection, we work in the following setup.

\begin{setup}\label{setup}
Let $X$ be a complex manifold of dimension $n$,
let $p\colon E\to X$ be a vector bundle of rank $r$,
and let $\varphi\colon E\to T_X$ be a bundle morphism,
not necessarily surjective.
We write $0_x$ for the zero vector in $E_x$ and $0_X\subset E$ for the image of the zero section.
For $t\in\mathbb C$, let $m_t\colon E\to E$ denote fibrewise multiplication by $t$.
\end{setup}

We first define quadratic $\varphi$-vector fields.

\begin{definition}\label{def:q-spray}
A \emph{quadratic $\varphi$-vector field}
is a vector field $Z$
on the total space of $E$ satisfying the following conditions:
\begin{itemize}
\item[$\bullet$] $dp_v(Z(v))=\varphi_x(v)
\quad (x\in X,\ v\in E_x).$

\item[$\bullet$] $Z(tv)=t\,(dm_t)_vZ(v)
\quad (v\in E,\ t\in\mathbb C).$
\end{itemize}
Here $(dm_t)_v\colon T_vE\to T_{tv}E$ denotes the differential of $m_t$ at $v\in E$.
\end{definition}

\begin{remark}\label{rem:APS}
The definition is modelled on sprays in differential geometry
in the sense of Ambrose, Palais and Singer \cite{AmbrosePalaisSinger1960}.
Indeed, for $E=T_X$ and $\varphi=\id_{T_X}$, the first condition says
that $Z$ is a second-order vector field, and the second one is the
homogeneity condition in \cite{AmbrosePalaisSinger1960}; thus a quadratic
$\id_{T_X}$-vector field is a spray  on $X$ in their sense.
The exponential $\varphi$-spray in Proposition~\ref{prop-integral}
below corresponds to the exponential map of such a spray.
\end{remark}

We show that a quadratic $\varphi$-vector field on $E$
integrates $\varphi\colon E\to T_X$ to a local spray.

\begin{proposition}[Integration of quadratic $\varphi$-vector fields]
\label{prop-integral}
Every quadratic $\varphi$-vector field $Z$
on $E$ is integrable; that is, there exist an open neighbourhood $U\subset E$ of $0_X$ and
a holomorphic map $s\colon U\to X$ such that
\[
s(0_x)=x,
\qquad
\left.\frac{d}{d\lambda}\right|_{\lambda=0}s(\lambda v)
=\varphi_x(v)
\qquad(x\in X,\ v\in E_x).
\]
More precisely, such a map $s \colon U \to X$ is given by
$s=p\circ\operatorname{Fl}^Z_1$, where $\operatorname{Fl}^Z_t$ denotes the flow of $Z$;
we call it the \emph{exponential $\varphi$-spray} of $Z$.
In particular, if $\varphi$ is surjective, then $s$ is a local dominating spray.
\end{proposition}

\begin{proof}

\smallskip
\noindent\emph{Step 1. Local expression of quadratic $\varphi$-vector fields.}
We first express the defining conditions of quadratic $\varphi$-vector fields in local coordinates.
Take an open set $V\subset X$ and a frame $e=(e_{1},\ldots,e_{r})$ of $E|_{V}$, regarded as a row vector.
This frame defines the trivialization
\[
V\times\mathbb C^r\longrightarrow E|_{V},
\qquad (x,w)\longmapsto e(x)w
\]
and the natural isomorphism 
$$
T_{e(x)w}E \cong T_xX\oplus\mathbb C^r. 
$$
With respect to this decomposition, the restriction of $Z$ to $E|_{V}$ can be written as
\[
Z(x,w)=\bigl(A(x,w),B(x,w)\bigr),
\]
where $A(x,w)\in T_xX$ is its base component and
$B(x,w)\in\mathbb C^r$ is its fibre component.

Since $p(x,w)=x$ in this trivialization, we have $dp_{(x,w)}(c,d)=c$.
Thus, for $v=e(x)w$, the first condition $dp_v(Z(v))=\varphi_x(v)$
is equivalent to
\begin{align}\label{eq-1}
A(x,w)=\varphi_x(e(x)w).
\end{align}
In particular, the base component is uniquely determined by
$\varphi$ and is linear in $w$.

The fibrewise multiplication $m_t$ is given by $m_t(x,w)=(x,tw)$,
so its differential is
$(dm_t)_{(x,w)}(c,d)=(c,td)$.
Consequently, we obtain 
\[
t\,(dm_t)_{(x,w)}Z(x,w)
=\bigl(tA(x,w),t^2B(x,w)\bigr).
\]
Thus, the second condition $Z(tv)=t\,(dm_t)_vZ(v)$ is equivalent to
\[
A(x,tw)=tA(x,w) \text{ and }
B(x,tw)=t^2B(x,w).
\]
The first equality already follows from the expression \eqref{eq-1} for $A$.
The second requires the fibre component to be homogeneous of degree
two: one factor of $t$ comes from the differential $dm_t$, and the
other from the scalar multiplying this differential.
This homogeneity is equivalent to the
existence of a symmetric bilinear map
$\gamma(x)\colon\mathbb C^r\times\mathbb C^r\to\mathbb C^r$,
depending holomorphically on $x$, such that
\begin{align}\label{eq-2}
B(x,w)=-\gamma(x)(w,w).
\end{align}
The minus sign is a convention chosen for the subsequent calculations.

\smallskip
\noindent\emph{Step 2. Construction of the local spray by the flow.}
Consider the local flow $\operatorname{Fl}^Z_t$ of $Z$,
characterized by
\[
\frac{d}{dt}\operatorname{Fl}^Z_t(v)
=
Z\bigl(\operatorname{Fl}^Z_t(v)\bigr),
\qquad
\operatorname{Fl}^Z_0(v)=v.
\]
The local existence of the flow and its holomorphic dependence
on the initial point $v$ follow from the theory of holomorphic
ordinary differential equations.

We show that the flow is defined for $0\le t\le1$
on a sufficiently small neighbourhood of the zero section.
Fix $x\in X$.
There exist an open neighbourhood $W$ of $0_x$ and $\varepsilon>0$
such that the local flow is defined for initial points in $W$ and $|t|<\varepsilon$.
The local expressions \eqref{eq-1} and \eqref{eq-2} show that
$Z(0_x)=0$ for every $x\in X$.
Hence the solution starting at $0_x$ is constant.
Thus, for an integer $N\in\mathbb Z_+$ with $1/N<\varepsilon$,
the holomorphic map $\operatorname{Fl}^Z_{1/N}$ fixes $0_x$.
After shrinking the neighbourhood of $0_x$, its $N$-fold composition
is well defined.
By uniqueness of solutions, this composition is the time-one map
$\operatorname{Fl}^Z_1$.
Applying this construction at every point of the zero section and
using uniqueness on overlaps gives a holomorphic map
\[
\operatorname{Fl}^Z_1\colon U\longrightarrow E
\]
on an open neighbourhood $U$ of $0_X$.
Since the neighbourhood is allowed to depend on $x$.

We show that the holomorphic map
$$s:=p\circ\operatorname{Fl}^Z_1\colon U\to X$$
is a local spray whose fibre derivative coincides with $\varphi$.
Since the flow fixes the zero section, we have $s(0_x)=x$.
It remains to compute its fibre derivative.
Fix $x\in V$ and $v=e(x)w\in E_x$.
For $\lambda\in\mathbb C$ sufficiently close to zero, the orbit 
starting at $\lambda v$ remains in this coordinate neighbourhood for
$0\le t\le1$. We write
\[
\operatorname{Fl}^Z_t(\lambda v)
=
\bigl(z(t,\lambda),w(t,\lambda)\bigr).
\]
By \eqref{eq-1} and \eqref{eq-2}, the flow equations are
\begin{equation}\label{eq-3}
\begin{aligned}
\frac{\partial z}{\partial t}
&=\varphi(z(t,\lambda))w(t,\lambda),\\
\frac{\partial w}{\partial t}
&=-\gamma(z(t,\lambda))
  \bigl(w(t,\lambda),w(t,\lambda)\bigr),
\end{aligned}
\end{equation}
with initial conditions
\[
z(0,\lambda)=z(x),
\qquad
w(0,\lambda)=\lambda w.
\]
Here $\varphi(z)$ denotes the matrix of $\varphi$ with respect to the frame $e$
and the base coordinates $z$, and $\gamma(z)$ denotes the corresponding
local coefficient expression of the quadratic term.

Differentiate the second equation in
\eqref{eq-3} with respect to $\lambda$ and set
$\lambda=0$.
The terms obtained by differentiating the coefficients of $\gamma$ contain two factors
$w(t,0)$, while those obtained by differentiating either argument of $\gamma$ contain
one such factor.
All these terms vanish, so we obtain 
\[
\frac{d}{dt}
\left(
\left.\frac{\partial w(t,\lambda)}{\partial\lambda}\right|_{\lambda=0}
\right)=0.
\]
Using the initial condition, we obtain
\[
\left.\frac{\partial w(t,\lambda)}{\partial\lambda}\right|_{\lambda=0}
=w
\qquad(0\le t\le1).
\]

Next, differentiate the first equation in
\eqref{eq-3} with respect to $\lambda$.
The term obtained by differentiating the coefficients of $\varphi$
is multiplied by $w(t,0)=0$ and therefore vanishes.
The remaining term is $\varphi(z(x))w$, the coordinate
expression of $\varphi_x(v)$.
Thus, identifying tangent vectors with their coordinate expressions, we obtain 
\[
\frac{d}{dt}
\left(
\left.\frac{\partial z(t,\lambda)}{\partial\lambda}\right|_{\lambda=0}
\right)
=\varphi_x(v).
\]
Since $z(0,\lambda)$ is independent of $\lambda$, integration gives
\[
\left.\frac{\partial z(t,\lambda)}{\partial\lambda}\right|_{\lambda=0}
=t\,\varphi_x(v).
\]
Evaluating at $t=1$ yields
\[
\begin{aligned}
\left.\frac{d}{d\lambda}\right|_{\lambda=0}s(\lambda v)
=(dp)_{0_x}
  \bigl((d\operatorname{Fl}^Z_1)_{0_x}(0,v)\bigr)
=(dp)_{0_x}\bigl(\varphi_x(v),v\bigr)
=\varphi_x(v)
\end{aligned}
\]
as required.
\end{proof}

We now define the cohomology class appearing in the criterion.

\begin{definition}\label{def:atiyah}
The \emph{symmetrized $\varphi$-Atiyah class} of $E$ is defined by
\[
\alpha_\varphi(E)
:=\operatorname{Sym}\bigl(\varphi^*a(E)\bigr)
\in H^1(X,S^2E^\vee\otimes E),
\]
where $a(E)\in H^1(X,\Omega_X^1\otimes\operatorname{End}E)$
is the Atiyah class of $E$ and $\varphi^*$ denotes the map on cohomology induced by
$\varphi^*\colon\Omega_X^1\to E^\vee$, so that
\[
\varphi^*a(E)
\in H^1(X,E^\vee\otimes\operatorname{End}E)
\cong H^1(X,E^\vee\otimes E^\vee\otimes E).
\]
The morphism $\operatorname{Sym}$ symmetrizes the two factors $E^\vee$:
for an $E$-valued bilinear form $B$, we set
\[
\operatorname{Sym}B(u,v)
:=\frac12\bigl(B(u,v)+B(v,u)\bigr).
\]
\end{definition}

\begin{theorem}[Cohomological criterion for quadratic $\varphi$-vector fields]
\label{thm:criterion}
The following conditions are equivalent:
\begin{enumerate}
\item[$(1)$] The bundle $E$ admits a quadratic $\varphi$-vector field.
\item[$(2)$] The symmetrized $\varphi$-Atiyah class
$\alpha_\varphi(E)$ vanishes.
\end{enumerate}
\end{theorem}

\begin{proof}
Every local trivialization of $E$ yields a local quadratic $\varphi$-vector field.
We show that local quadratic $\varphi$-vector fields can be chosen so that they glue together
precisely when a cocycle representing $\alpha_\varphi(E)$ is a coboundary.

\smallskip
\noindent\emph{Step 1. A cocycle representing $\alpha_\varphi(E)$.}
Let $\{V_i\}$ be a locally finite open cover of $X$ by Stein open sets,
and let $e_i:=(e_{i,1},\ldots,e_{i,r})$ be a frame of
$E|_{V_i}$.
Let $\{g_{ij}\}$ denote the transition matrices of $E$, defined by $e_i=e_jg_{ij}$.
The fibre coordinates $w_i\colon E|_{V_i}\to\mathbb C^r$ are the
column vectors defined by
\[
v=e_i(x) \cdot w_i(v)\qquad(v\in E_x,\ x\in V_i).
\]
Then
\begin{equation}\label{eq:quadratic-fibre-coordinates}
w_j=(g_{ij}\circ p)w_i \quad \text{ on } E|_{V_i\cap V_j}.
\end{equation}

Let $\nabla^i$ be the holomorphic connection on $E|_{V_i}$ that differentiates
the coefficients of a section in the frame $e_i$:
\[
\nabla^i(e_i u_i)=e_i\,du_i.
\]
On $V_{ij}:=V_i\cap V_j$, the Leibniz rule shows that
$A_{ij}:=\nabla^j-\nabla^i$ is $\mathcal O_X$-linear, so
$A_{ij}\in H^0(V_{ij},\Omega_X^1\otimes\operatorname{End}(E))$.
With our sign convention, the cocycle $\{A_{ij}\}$ represents the Atiyah class $a(E)$.
For a section $u=e_i u_i=e_j u_j$, we have $u_j=g_{ij}u_i$.
Consequently,
\[
\nabla^j u
=e_j\,d(g_{ij}u_i)
=e_i\bigl(du_i+g_{ij}^{-1}dg_{ij}\,u_i\bigr).
\]
Subtracting $\nabla^i u=e_i\,du_i$, we obtain
\begin{equation}\label{eq:quadratic-atiyah-cocycle}
A_{ij} = g_{ij}^{-1}dg_{ij},
\end{equation}
where we identify $A_{ij}$ with its matrix of holomorphic $1$-forms in the frame $e_i$.

Pulling back by $\varphi$ and symmetrizing, we obtain the cocycle
\[
C_{ij}(u,v)
:=\frac12\bigl(
A_{ij}(\varphi(u))v+A_{ij}(\varphi(v))u
\bigr)
\qquad (u,v\in E_x,\ x\in V_{ij}). 
\]
By definition, the symmetrized $\varphi$-Atiyah class
$\alpha_\varphi(E)\in H^1(X,S^2E^\vee\otimes E)$ is represented by this cocycle.
Here $A_{ij}(\varphi(u))$ denotes the endomorphism of $E_x$ obtained by evaluating
the $1$-form factor of $A_{ij}$ on $\varphi(u)$.
In particular,
$
C_{ij}(v,v)=A_{ij}(\varphi(v))v
$.

\smallskip
\noindent\emph{Step 2. A gluing condition for quadratic $\varphi$-vector fields.}
Given a section $\Gamma_i\in H^0(V_i,S^2E^\vee\otimes E)$,
we define a symmetric bilinear map
$\gamma_i(x)\colon\mathbb C^r\times\mathbb C^r\to\mathbb C^r$
by 
\[
\Gamma_{i,x}(e_i(x)a,e_i(x)b)
=e_i(x)\gamma_i(x)(a,b).
\]
Using $\gamma_i$,
we then define a local quadratic $\varphi$-vector field $Z_i$ on $E|_{V_i}$
by \eqref{eq-1} and \eqref{eq-2}.

We show that the vector fields $Z_i$ glue together to a global vector field on $E$
if and only if the sections $\Gamma_i$ satisfy
\begin{equation}\label{eq:quadratic-coboundary}
\Gamma_i-\Gamma_j=C_{ij}
\end{equation}
on every overlap.

To verify this, we compare the actions of $Z_i$ and $Z_j$ on holomorphic functions.
By \eqref{eq-1}, $Z_i$ and $Z_j$ agree on functions pulled back from $X$.
It remains to compare their actions on the components of $w_j$.
Fix $x\in V_i\cap V_j$ and $v\in E_x$.
Applying $Z_i$ to \eqref{eq:quadratic-fibre-coordinates} and using
the Leibniz rule, we obtain
\begin{equation}\label{eq:quadratic-transition-action}
\begin{aligned}
Z_i(w_j)(v)
&=Z_i(g_{ij}\circ p)(v)w_i(v) +g_{ij}(x)Z_i(w_i)(v)\\
&=(dg_{ij})_x(\varphi_x(v))w_i(v)
-g_{ij}(x)\gamma_i(x)(w_i(v) ,w_i(v) )
\end{aligned}
\end{equation}
by \eqref{eq-1} and \eqref{eq-2}. 
Here the vector field acts componentwise on the matrix $g_{ij}\circ p$
and on the coordinate maps $w_i,w_j$.
On the other hand, \eqref{eq-2} gives
$Z_j(w_j)(v)=-\gamma_j(x)(w_j(v),w_j(v))$.
Hence $Z_i=Z_j$ if and only if
\[
g_{ij}(x)\gamma_i(x)(w_i(v),w_i(v))-\gamma_j(x)(w_j(v),w_j(v))
=(dg_{ij})_x(\varphi_x(v))w_i(v)
\]
for every $v\in E_x$.
Multiplying this identity by the frame $e_j(x)$ turns it into
an identity in $E_x$,
whose left-hand side becomes
$\Gamma_{i,x}(v,v)-\Gamma_{j,x}(v,v)$ since $e_jg_{ij}=e_i$.
Thus we obtain
\[
\begin{aligned}
\Gamma_{i,x}(v,v)-\Gamma_{j,x}(v,v)
&=e_j(x)(dg_{ij})_x(\varphi_x(v))w_i(v) \\
&=e_i(x)g_{ij}(x)^{-1}(dg_{ij})_x(\varphi_x(v))w_i(v)\\
&=A_{ij,x}(\varphi_x(v))v\\
&=C_{ij,x}(v,v).
\end{aligned}
\]
Therefore, the vector fields $Z_{i}$ and $Z_{j}$ agree precisely when
\[
\Gamma_{i,x}(v,v)-\Gamma_{j,x}(v,v)=C_{ij,x}(v,v)
\qquad(v\in E_x).
\]
Since both sides are symmetric bilinear tensors, equality on all pairs
$(v,v)$ implies equality on all pairs by polarization.
This proves the gluing criterion.

\smallskip
\noindent\emph{Step 3. Conclusion.}
Suppose that $(2)$ holds.
Since $\{C_{ij}\}$ represents the zero class,
there exist sections $\Gamma_i$
such that $\Gamma_i-\Gamma_j=C_{ij}$.
Let $\gamma_i$ be their coefficient expressions in the frames $e_i$,
and define $Z_i$ by \eqref{eq-1} and \eqref{eq-2}.
By Step~2, these vector fields glue together to a vector field
$Z$ on $E$,
which is a quadratic $\varphi$-vector field by its local expression.
This proves $(1)$.

Conversely, suppose that $(1)$ holds, and let $Z$ be a quadratic
$\varphi$-vector field.
Write $Z$ locally as in \eqref{eq-1} and \eqref{eq-2}, and let
$\Gamma_i$ be the corresponding tensors.
These local expressions agree on overlaps since they come from
the same global vector field.
By Step~2, we have $\Gamma_i-\Gamma_j=C_{ij}$.
Hence $\{C_{ij}\}$ is a coboundary, so
$\alpha_\varphi(E)=0$, which proves $(2)$.
\end{proof}

\begin{corollary}\label{cor:criterion}
If $\varphi \colon E \to T_{X}$ is a surjective bundle morphism
and the symmetrized $\varphi$-Atiyah class
$\alpha_\varphi(E)$ vanishes, 
then there exists a local dominating spray $s \colon U \to X$.
\end{corollary}
\begin{proof}
By Theorem~\ref{thm:criterion}, $E$ admits a quadratic
$\varphi$-vector field.
Its exponential $\varphi$-spray is a local spray
with fibre derivative $\varphi$ by Proposition~\ref{prop-integral}.
Since $\varphi$ is surjective, this spray is dominating.
\end{proof}

\begin{remark}\label{rem:quadratic-versus-local}
The vanishing of $\alpha_\varphi(E)$ is necessary and sufficient
for the existence of a quadratic $\varphi$-vector field, but it
is only a sufficient condition for the existence of a local spray with fibre derivative $\varphi$ (not necessarily an exponential $\varphi$-spray).

Indeed, let $X=\mathbb C/\Lambda$ be an elliptic curve, let $L$
be a line bundle on $X$ with $\deg L\ne0$, and set
\[
E:=T_X\oplus L,
\qquad
\varphi\colon E\longrightarrow T_X,
\qquad \varphi(u,\ell)=u.
\]
The translation trivialization $T_X\cong X\times\mathbb C$
defines a global dominating spray
\[
s\colon E\longrightarrow X,
\qquad s(x,u,\ell)=x+u\pmod{\Lambda},
\]
whose fibre derivative is $\varphi$.
We show that $\alpha_\varphi(E)\ne0$.

Choose a global translation-invariant frame of $T_X$ and local frames $e_i^L$ of $L$, with transition functions $h_{ij}$
defined by $e_i^L=e_j^Lh_{ij}$.
The corresponding transition matrices of $E$ and the cocycle
representing its Atiyah class are
\[
g_{ij}=
\begin{pmatrix}
1&0\\
0&h_{ij}
\end{pmatrix},
\qquad
A_{ij}=g_{ij}^{-1}dg_{ij}
=\begin{pmatrix}
0&0\\
0&A^L_{ij}
\end{pmatrix},
\qquad A^L_{ij}:=h_{ij}^{-1}dh_{ij}.
\]
Here $\{A^L_{ij}\}$ represents
$a(L)\in H^1(X,\Omega_X^1)$, where we use
$\operatorname{End}L\cong\mathcal O_X$.
In particular, for $\xi\in T_xX$ and $(u,\ell)\in E_x$,
\[
A_{ij,x}(\xi)(u,\ell)
=\bigl(0,A^L_{ij,x}(\xi)\ell\bigr).
\]
As in the proof of Theorem~\ref{thm:criterion}, the class
$\alpha_\varphi(E)$ is represented by
\[
C_{ij}(a,b)
=\frac12\bigl(
A_{ij}(\varphi(a))b+A_{ij}(\varphi(b))a
\bigr).
\]
For $u\in T_xX$ and $\ell\in L_x$, we have
$\varphi(u,0)=u$ and $\varphi(0,\ell)=0$.
Consequently,
\[
\begin{aligned}
C_{ij}((u,0),(0,\ell))
&=\frac12\bigl(
A_{ij}(u)(0,\ell)+A_{ij}(0)(u,0)
\bigr)\\
&=\left(0,\frac12 A^L_{ij}(u)\ell\right).
\end{aligned}
\]

To interpret this identity in cohomology, we use the direct sum
decomposition $E=T_X\oplus L$ to write
\[
S^2E^\vee
\cong S^2T_X^\vee
\oplus(T_X^\vee\otimes L^\vee)
\oplus(L^\vee)^{\otimes2}.
\]
Tensoring with $E=T_X\oplus L$, we obtain
\[
\begin{aligned}
S^2E^\vee\otimes E
\cong{}& (S^2T_X^\vee\otimes T_X)
\oplus(S^2T_X^\vee\otimes L)\\
&\oplus(T_X^\vee\otimes L^\vee\otimes T_X)
\oplus(T_X^\vee\otimes L^\vee\otimes L)\\
&\oplus((L^\vee)^{\otimes2}\otimes T_X)
\oplus((L^\vee)^{\otimes2}\otimes L).
\end{aligned}
\]
The fourth summand consists of the mixed terms
with values in $L$ and is canonically isomorphic to
$\Omega_X^1$ since $L^\vee\otimes L\cong\mathcal O_X$.
Under this identification, a $1$-form $\omega$ corresponds to
the bilinear map $(u,\ell)\mapsto\omega(u)\ell$.
The preceding calculation shows that the component of $C_{ij}$
in this summand is $\frac12 A^L_{ij}$.
Thus, in the induced direct sum decomposition of
$H^1(X,S^2E^\vee\otimes E)$, the component of
$\alpha_\varphi(E)$ in
\[
H^1(X,T_X^\vee\otimes L^\vee\otimes L)
\cong H^1(X,\Omega_X^1)
\]
is $(1/2) a(L)$.
Since $a(L)\ne0$ by $\deg L\ne0$, it follows that $\alpha_\varphi(E)\ne0$.
By Theorem~\ref{thm:criterion}, $E$ admits
no quadratic $\varphi$-vector field, and hence no exponential
$\varphi$-spray, although it admits the
global dominating spray defined above.
\end{remark}

\subsection{Examples admitting no local dominating spray}
\label{subsec:applications}

Theorem~\ref{thm:weakly} shows that every
weakly pseudoconvex manifold admitting a positive line bundle
has a local dominating spray.
However, a compact K\"ahler manifold
without a positive line bundle need not admit a local dominating spray. 
Thus the positivity hypothesis cannot be omitted in general.

\begin{example}[Compact K\"ahler manifolds without local dominating sprays]
\label{ex:torus-blowup-no-local-spray}
Let $\mathbb T$ be a complex torus of algebraic dimension zero, and let
$\pi\colon X:=\operatorname{Bl}_{t_0}(\mathbb T)\to\mathbb T$
be the blow-up at a point $t_0\in\mathbb T$.
Then $X$ admits no local dominating spray.
\end{example}

\begin{proof}
After a translation of $\mathbb T$, we may assume that $t_0=0$.
Write $\mathbb T=\mathbb C^n/\Lambda$; note that $n\geq2$,
since complex tori of dimension one are projective. Let
$\Exc(\pi):=\pi^{-1}(0)$ be the exceptional divisor.
Suppose that a vector bundle $p\colon E\to X$
carries a local spray $s\colon U\to X$ whose fibre derivative
$\varphi\colon E\to T_X$ is surjective.
Set $r:=\operatorname{rank}E\geq n$ and
$X^\circ:=X\setminus\Exc(\pi)$.

We outline the argument.
We consider a \textit{displacement} $\delta$ defined by the dominating spray $s$. 
The coefficients of the Taylor expansion of $\delta$ along the fibres
impose a rank bound along the fibres of an algebraic reduction of $\mathbf P(E)$ (see Step~4).
On the other hand, we construct a hypersurface $S$ whose image $s(S)$ is contained in $\Exc(\pi)$ (see Step~5).
On $S$, the displacement $\delta(v)$ depends only on the base point $p(v)$,
and its differential has rank $n$ in the same directions, which gives a contradiction (see Step~8).

\smallskip
\noindent\emph{Step 1. The displacement and its Taylor coefficients.}
Via the natural quotient map $\mathbb C^n\to\mathbb T=\mathbb C^n/\Lambda$,
we identify a small neighbourhood of the origin in $\mathbb T$
with a neighbourhood of the origin in $\mathbb C^n$,
and we regard the points of this neighbourhood as vectors.
Since $s(0_x)=x$, we may shrink $U$ to a fibrewise ball bundle
such that $\pi(s(v))-\pi(p(v))$ lies in this origin chart
for every $v\in U$. 
We define the displacement
\[
\delta(v):=\pi(s(v))-\pi(p(v))\in\mathbb C^n 
\qquad (v\in U \subset E),
\]
where the difference is expressed in the fixed origin coordinates.
Since $\delta$ vanishes on the zero section, its fibrewise
Taylor expansion is
\begin{equation}\label{eq:torus-blowup-taylor}
\delta(v)=\sum_{k\geq1}\delta_k(v),\qquad
\delta_k(v):=\frac{1}{k!}
\left.\frac{d^k}{dt^k}\right|_{t=0}\delta(tv).
\end{equation}
The derivatives are taken componentwise in $\mathbb C^n$.
For every $v\in E_x$, the formula defining $\delta_k(v)$
makes sense since $tv\in U$ for sufficiently small $t$.
Each $\delta_k$ is a homogeneous polynomial of degree $k$
on every fibre and corresponds to a global section
of $S^kE^\vee\otimes(X\times\mathbb C^n)$.
The series in \eqref{eq:torus-blowup-taylor} converges locally uniformly on $U$.

\smallskip
\noindent\emph{Step 2. A coordinate map near the exceptional divisor.}
Choose $\rho>0$ so small that $B_{2\rho}\subset\mathbb C^n$
lies in the origin chart fixed above, and regard $B_\rho$
as an open subset of $\mathbb T$.
Define
$$
Q_0\colon N_0:=\pi^{-1}(B_\rho) \xrightarrow{\quad \pi \quad } B_\rho \subset \mathbb C^n. 
$$ 
Then $Q_0$ vanishes on $\Exc(\pi)$ and has differential
of rank $n$ on $N_0\setminus\Exc(\pi)$.
For $x,y\in N_0$, the torus difference $\pi(y)-\pi(x)$
is represented in the origin chart by $Q_0(y)-Q_0(x)$.

\smallskip
\noindent\emph{Step 3. Fixing the algebraic parameter.}
Let $q\colon\mathbf P(E)\to X$ be the bundle of lines in $E$,
and let $\nu\colon E\setminus0_X\to\mathbf P(E)$ be the
natural projection, so that $q\circ\nu=p$.
Choose an algebraic reduction
$\alpha\colon\mathbf P(E)\dashrightarrow Z$, 
where $Z$ is a smooth projective variety.
Every meromorphic function on $\mathbf P(E)$ is the pullback
under $\alpha$ of a rational function on $Z$ (see \cite{Ueno1975LNM}).
The relevant maps are shown in the following diagram:
\[
\xymatrix@C=3em@R=2.5em{
  \mathbf P(E) \ar[d]_{q} \ar@{-->}[rr]_{\alpha} & & Z \\
  X \ar[d]_{\pi}
  & E \ar[l]^{p}
  & E\setminus0_X \ar@{_{(}->}[l] \ar[ull]^{\nu} \\
  \mathbb T
  & U \ar@{^{(}->}[u] \ar[ul]^{s}
  &
}
\]

We show that
$(\pi\circ q,\alpha)\colon\mathbf P(E)\dashrightarrow
\mathbb T\times Z$ is a dominant meromorphic map.
Let $G\subset\mathbb T\times Z$ be the closure of its image,
obtained from the closure of its graph.
It is an irreducible compact analytic subset surjecting onto
both factors.
Over a dense Zariski open subset $V\subset\mathbb T$,
the projection $G\to\mathbb T$ is flat, and its fibres
$G_t\subset Z$ have a common Hilbert polynomial $Q$.
They therefore define a meromorphic map
$\mathbb T\dashrightarrow \operatorname{Hilb}^{Q}(Z)$.
Since the target is projective and $a(\mathbb T)=0$,
this map is constant.
Thus $G_t=Z'$ for every $t\in V$, for some fixed closed
algebraic subset $Z'\subset Z$.
Since $G$ is irreducible, $G\cap(V\times Z)$ is dense in $G$.
Taking closures, we obtain
\[
G=\overline{G\cap(V\times Z)}
=\overline{V\times Z'}
=\mathbb T\times Z.
\]

Hence the differential of $(\pi\circ q,\alpha)$ is surjective at a general point of $\mathbf{P}(E)$.
Since $d\pi$ is an isomorphism on $X^\circ$, we obtain
\begin{equation}\label{eq:torus-blowup-dominance}
dq_{[\ell]}\bigl(\ker d\alpha_{[\ell]}\bigr)
=T_{q([\ell])}X
\end{equation}
at a general point $[\ell] \in \mathbf{P}(E)$.
Thus the base point in $X$ can move in all $n$ tangent directions
while the value of $\alpha$ remains fixed.

\smallskip
\noindent\emph{Step 4. The rank bound from the Taylor coefficients.}
For $v\in E_x$, differentiating the displacement at $0_x$
in the fibre direction gives
\[
\delta_1(v)
=\left.\frac{d}{dt}\right|_{t=0}
  \bigl(\pi(s(tv))-\pi(x)\bigr)
=d\pi_x\bigl(\varphi_x(v)\bigr).
\]
Here the last term is identified with a vector in $\mathbb C^n$
via translation from $\pi(x)$ to the origin and the fixed
origin coordinates.
Since $\varphi_x$ is surjective and $d\pi_x$ is an isomorphism
for $x\in X^\circ$, the linear map
$\delta_1|_{E_x}\colon E_x\to\mathbb C^n$ is surjective for $x \in X^\circ$.
Hence its first component $\gamma:=\delta_1^1$
is a nonzero section of $E^\vee$. 

Work on the dense open subset of $\mathbf P(E)$ where
$\alpha$ is a holomorphic submersion, the base point lies
in $X^\circ$, and $\gamma$ is nonzero.
For $k\geq1$, each component of $\delta_k/\gamma^k$
is homogeneous of degree zero and defines a meromorphic
function on $\mathbf P(E)$.
Each component can be written as the pullback under $\alpha$ of a rational function on $Z$.
Thus its differential annihilates $\ker d\alpha$, first where
the rational function is holomorphic and then on the whole
open subset by holomorphic continuation.

Pulling back by $\nu$ and differentiating the ratio, we obtain 
\[
\begin{aligned}
(d\delta_k)_\ell(\eta)
&=d\left[
\gamma^k\left(\frac{\delta_k}{\gamma^k}\circ\nu\right)
\right]_\ell(\eta)\\
&=k\frac{\delta_k(\ell)}{\gamma(\ell)}
\,d\gamma_\ell(\eta) 
\quad 
(\eta\in\ker d(\alpha\circ\nu)_\ell).
\end{aligned}
\]
Here we regard $\delta_{k}$ as a holomorphic map $E \to \mathbb{C}^{n}$. 
Thus all $d\delta_k$ vanish on
$\ker d\gamma\cap\ker d(\alpha\circ\nu)$.
Differentiating \eqref{eq:torus-blowup-taylor} term by term, 
we deduce that 
\[
\operatorname{rank}
\left( (d\delta)_\ell\big|_{\ker d(\alpha\circ\nu)_\ell}\right)
\leq1
\]
since $d\gamma$ is scalar-valued. 

\smallskip
\noindent\emph{Step 5. A hypersurface mapped into the exceptional divisor.}
Fix $x_0\in\Exc(\pi)$, choose local coordinates
$z=(z_1,\ldots,z_n)$ centred at $x_0$ with
$\Exc(\pi)=\{z_1=0\}$, and trivialize $E$ with fibre
coordinates $w=(w_1,\ldots,w_r)$.
By surjectivity of $\varphi_{x_0}$, we may assume that
$\partial s^1/\partial w_1(0,0)\neq0$.
Shrink the product neighbourhood in $E$ so that its image under $s$
lies in this coordinate chart in $N_0$ and this derivative is nonzero there.

For $b=(b_2,\ldots,b_r)$ near zero, the implicit function theorem shows that
the equation
\[
s^1\bigl(z,t(1,b_2,\ldots,b_r)\bigr)=0
\]
has a unique small solution $t=t(z,b)$, which depends holomorphically on $(z,b)$.
Since $s^1(z,0)=z_1$, the solution $t(z,b)$ vanishes exactly when
$z_1=0$.
For $z_1\neq0$, we consider the holomorphic map
\[
(z,b)\longmapsto
\bigl(z,t(z,b)(1,b_2,\ldots,b_r)\bigr) \in E
\]
and define $S$ to be its image.
Thus $S$ is a local hypersurface in $E$ satisfying
$s(S)\subset\Exc(\pi)$ and $p(S)\subset X^\circ$.
In the projective chart $w_1\neq0$, the restriction
$\nu|_S$ is given by $(z,b)\mapsto(z,b)$.
Hence $\nu(S)$ is open in $\mathbf P(E)$, and
$\nu|_S\colon S\to\nu(S)$ is biholomorphic.

The solutions $t(z,0)$ tend to zero as $z\to x_0$.
Thus we may choose a small simply connected neighbourhood
$V\subset N_0\setminus\Exc(\pi)$ and a fixed fibre polydisc
$D_\varepsilon^r$ such that
\[
\mathcal A:=V\times D_\varepsilon^r\subset U,\qquad
s(\mathcal A)\subset N_0,\qquad S\cap\mathcal A\neq\varnothing.
\]
Replace $S$ by a nonempty open set contained in $\mathcal A$.
The key point is that the same connected product neighbourhood
$\mathcal A$ contains both the zero section over $V$ and
this nonempty patch of $S$.
Although this choice of $\mathcal A$ is not essential in the
torus blow-up case, we make it here in preparation for the next example.

\smallskip
\noindent\emph{Step 6. The displacement on the hypersurface.}
Since both $p(v)$ and $s(v)$ lie in $N_0$ for $v\in\mathcal A$,
the coordinate identity from Step~2 gives
$
\delta|_{\mathcal A}=Q_0\circ s-Q_0\circ p
$.
On $S$, the first term vanishes, so
\[
\delta|_S=-Q_0\circ p|_S.
\]
Thus the displacement on $S$ depends only on the base point $p(v)$.

\smallskip
\noindent\emph{Step 7. Extending the rank bound by analytic continuation.}
Put
$\mathcal A^\times:=\mathcal A\setminus0_X$.
In the chosen trivialization,
$\mathcal A^\times=V\times(D_\varepsilon^r\setminus\{0\})$,
so $\mathcal A^\times$ is connected.
By Step~4, we have
$
\operatorname{rank}\left(
d\delta\big|_{\ker d(\alpha\circ\nu)}
\right)\leq1
$
on a nonempty open subset of $\mathcal A^\times$.
By analytic continuation, this inequality holds wherever
$\alpha\circ\nu$ is a holomorphic submersion.

\smallskip
\noindent\emph{Step 8. The differential has rank $n$ on $S$.}
Since $\nu(S)$ is open in $\mathbf P(E)$, we can choose $\ell\in S$
such that $\alpha$ is a holomorphic submersion at
$[\ell]:=\nu(\ell)$ and \eqref{eq:torus-blowup-dominance}
holds there.
The map $\nu|_S$ is locally biholomorphic and identifies
$T_\ell S\cap\ker d(\alpha\circ\nu)_\ell$
with $\ker d\alpha_{[\ell]}$.
Since $p=q\circ\nu$, we obtain
\[
dp_\ell\left(
T_\ell S\cap\ker d(\alpha\circ\nu)_\ell
\right)
=T_{p(\ell)}X.
\]
By Step~6, we have $\delta=-Q_0\circ p$ on $S$.
Since $p(\ell)\in V\subset N_0\setminus\Exc(\pi)$ and
$(dQ_0)_{p(\ell)}$ has rank $n$, the preceding surjectivity
yields
\[
\operatorname{rank}\left(
(d\delta)_\ell
\big|_{T_\ell S\cap\ker d(\alpha\circ\nu)_\ell}
\right)=n.
\]
This contradicts the rank bound in Step~7, since $n\geq2$.
Thus $X$ admits no local dominating spray.
\end{proof}

\begin{example}[Kummer surfaces without local dominating sprays]
\label{ex:kummer-no-local-spray}
Let $\mathbb T=\mathbb C^2/\Lambda$ be a two-dimensional complex
torus of algebraic dimension zero, and let $X$ be its
Kummer surface. Then $X$ admits no local dominating spray.
\end{example}

\begin{proof}
We note that $\mathbb T$ contains no curves, and hence no
nontrivial subvarieties.
Indeed, let $C\subset\mathbb T$ be an irreducible curve.
Since $K_{\mathbb T}$ is trivial, the adjunction formula gives
$C^2=2p_a(C)-2$.
As $\mathbb T$ contains no rational curves, we have $p_a(C)\ge1$, and hence $C^2\ge0$.
If $C^2>0$, then the line bundle $\mathcal O_{\mathbb T}(C)$ is big,
and hence $\mathbb T$ is Moishezon.
If $C^2=0$, then $p_a(C)=1$ and $C$ is a smooth elliptic curve,
which is a translate of a one-dimensional subtorus $\mathbb T'$;
then the quotient map $\mathbb T\to\mathbb T/\mathbb T'$ onto an elliptic curve
yields nonconstant meromorphic functions on $\mathbb T$.
In either case, $a(\mathbb T)\ge1$, a contradiction.

Let $\mathbb T[2]:=(1/2)\Lambda/\Lambda$ be the set of the sixteen
two-torsion points, and let
$\pi\colon\widetilde{\mathbb T}
:=\operatorname{Bl}_{\mathbb T[2]}\mathbb T\to\mathbb T$
be the blow-up of $\mathbb T$ at these points.
The involution $y\mapsto-y$ lifts to an involution
$\widetilde\iota$ of $\widetilde{\mathbb T}$.
The quotient $X=\widetilde{\mathbb T}/\widetilde\iota$ is a
K3 surface, and the quotient map
$\tau\colon\widetilde{\mathbb T}\to X$ is a double cover
branched along the curves
$\Gamma_b:=\tau(\pi^{-1}(b))$.
Set $\Gamma:=\bigcup_{b\in \mathbb T[2]}\Gamma_b$ and $X^\circ:=X\setminus\Gamma$.

Fix $a\in\mathbb T[2]$.
Suppose that a vector bundle $p\colon E\to X$
of rank $r$ carries a local spray $s\colon U\to X$ whose
fibre derivative $\varphi\colon E\to T_X$ is surjective
at every point of $\Gamma_a$.

As in Example~\ref{ex:torus-blowup-no-local-spray}, we use
a displacement whose Taylor coefficients impose a rank bound
along the fibres of an algebraic reduction of $\mathbf P(E)$.
We then construct a hypersurface $S$ that $s$ maps into
$\Gamma_a$. Passing to a branched double cover makes the
displacement holomorphic over $S$, where its differential
has rank two in the same directions; this gives a contradiction.

\smallskip
\noindent\emph{Step 1. The squared displacement.}
Via the natural quotient map $\mathbb C^2\to\mathbb T$,
we identify a small neighbourhood of the origin in $\mathbb T$
with a neighbourhood of the origin in $\mathbb C^2$,
and we use this identification to regard small torus differences
as vectors.

The maps $\pi$ and $\tau$ induce a natural unramified double
cover $\tau^{\circ} \colon \mathbb T^\circ:=\mathbb T\setminus\mathbb T[2]\to X^\circ$, 
whose two sheets are interchanged by $y\mapsto-y$.

Take a fibrewise ball bundle $U^\circ \subset U \cap p^{-1}(X^\circ)$ 
such that $s(U^\circ)\subset X^\circ$.
For $x\in X^\circ$, we lift $s|_{U^\circ_x}$ to a holomorphic
map $\widetilde s_x\colon U^\circ_x\to\mathbb T^\circ$,
choosing one point of the preimage of $x$ as its initial value.
After a further fibrewise shrinking, we define a holomorphic map
\[
\delta\colon U^\circ\to S^2\mathbb C^2
\quad \text{ by } \delta(v):=
\bigl(\widetilde s_x(v)-\widetilde s_x(0_x)\bigr)^{\otimes2}
\in S^2\mathbb C^2,\quad (v\in U^\circ_x).
\]
Here the difference is expressed in the fixed coordinates
near the origin, and
$u^{\otimes2}=(u_1^2,u_1u_2,u_2^2)$ for $u=(u_1,u_2)$.
This $\delta$ is well defined and holomorphic.
Indeed, changing the initial value changes
the signs of both terms but leaves $\delta$ unchanged.
Moreover, the lifts vary holomorphically over small simply connected neighbourhoods in $X^\circ$, so these expressions
define a holomorphic map $\delta\colon U^\circ\to S^2\mathbb C^2$.

Since the difference vanishes on the zero section, its square
has the fibrewise Taylor expansion
\begin{equation}\label{eq:expansion}
\delta(v)=\sum_{k\geq2}\delta_k(v),\qquad
\delta_k(v):=\frac{1}{k!}
\left.\frac{d^k}{dt^k}\right|_{t=0}\delta(tv).
\end{equation}
Each $\delta_k$ is a homogeneous polynomial of degree $k$
on every fibre and corresponds to a section of
$S^kE^\vee\otimes(X^\circ\times S^2\mathbb C^2)$ over $X^\circ$.
The series in \eqref{eq:expansion}, together with its first
derivatives, converges locally uniformly on $U^\circ$.

\smallskip
\noindent\emph{Step 2. Meromorphic Taylor coefficients.}
In this step, we prove that $\delta_{k}z_{1}^{k}$ extends holomorphically across $\Gamma_b$,
where $z_{1}$ is a local defining function of $\Gamma_b$,
and hence that $\delta_{k}$ is a meromorphic section of $S^kE^\vee\otimes(X\times S^2\mathbb C^2)$.
To this end, we first describe the domain of $\delta$ precisely.

Choose $\rho>0$ so small that $B_{2\rho}\subset\mathbb C^2$
lies in the origin chart fixed above and the neighbourhoods
$b+B_\rho\subset\mathbb T$, $b\in\mathbb T[2]$, are pairwise disjoint.
For each $b$, we consider the
coordinate map 
$$u_b\colon b+B_\rho\to B_\rho \quad 
 \text{ defined by }b+u\mapsto u
$$ 
and put 
\[
\widetilde N_b:=\pi^{-1}(b+B_\rho),\qquad
N_b:=\tau(\widetilde N_b).
\]
The restriction of the natural double cover gives
$b+B_\rho\setminus\{b\}\to N_b\setminus\Gamma_b$.
For two points of $b+B_\rho$, their torus difference is
represented by the difference of their $u_b$-coordinates.

Near a point of $\Gamma_b$, choose coordinates
$z=(z_1,z_2)$ with $\Gamma_b=\{z_1=0\}$ and trivialize $E$
with fibre coordinates $w\in\mathbb C^r$.
Since $s(z,0)=z$, we have
$|s^1(z,w)-z_1|\leq C|w|$ for some uniform constant $C$.
Fix $c>0$ so small that, for $z_1\neq0$,
the full fibre ball
$D_z:=\{w:|w|<c|z_1|\}$ satisfies
\[
|s^1(z,w)|\geq\tfrac12|z_1|>0
\qquad(w\in D_z).
\]
Thus $s(z,D_z)$ avoids $\Gamma_b$, and the map $w\mapsto s(z,w)$ on $D_z$ lifts to $b+B_\rho\setminus\{b\}$
via the unramified cover $b+B_\rho\setminus\{b\}\to N_b\setminus\Gamma_b$;
we denote this lift by $\widetilde s_z$.
Near $w=0$, the displacement from Step~1 is given by
\[
\delta(z,w)=
\bigl(u_b(\widetilde s_z(w))
      -u_b(\widetilde s_z(0))\bigr)^{\otimes2}.
\]
The right-hand side is defined on all of $D_z$, independently
of the choice of lift, and each component is bounded by
$4\rho^2$, uniformly in $z$.

Consider a component 
$\delta_k^{ij}(z,w)=\sum_{|\mu|=k}a_\mu^{ij}(z)w^\mu$
of $\delta_{k}$, where $1\leq i\leq j\leq2$ and $\mu$ is a multi-index.
Since the closed fibre polydisc of radius
$c|z_1|/(2\sqrt r)$ lies in $D_z$,
Cauchy's estimates for the coefficients give
\[
|a_\mu^{ij}(z)|
\leq4\rho^2\left(\frac{2\sqrt r}{c|z_1|}\right)^k
\qquad(|\mu|=k).
\]
Consequently, by the Riemann extension theorem, $z_1^ka_\mu^{ij}$ extends holomorphically across $\Gamma_b$.
Thus each $\delta_k$ extends meromorphically across $\Gamma_b$
with poles of order at most $k$.
Note that only the individual Taylor coefficients are extended. 

\smallskip

For later use, we introduce the squared coordinate map
$$
Q_b\colon N_b\to S^2\mathbb C^2
\text{ defined by }Q_b\bigl(\tau(\widetilde y)\bigr)
:=
\bigl(u_b(\pi(\widetilde y))\bigr)^{\otimes2}. 
$$ 
It is well defined and holomorphic since $b=-b$ and
$u_b(-y)=-u_b(y)$.
In a blow-up chart, write $u_b\circ\pi=(t,t\zeta)$.
The quotient coordinates are then $z_1=t^2$ and $z_2=\zeta$, which gives
$Q_b(z_1,z_2)=(z_1,z_1z_2,z_1z_2^2)$.
This expression shows that $Q_b$ vanishes on $\Gamma_b$ and
has differential of rank two on $N_b\setminus\Gamma_b$.

\smallskip
\noindent\emph{Step 3. Fixing the algebraic parameter.}
Let $q\colon\mathbf P(E)\to X$ be the bundle of lines in $E$,
and let $\nu\colon E\setminus0_X\to\mathbf P(E)$ be the
natural projection, so that $q\circ\nu=p$.
Choose an algebraic reduction
$\alpha\colon\mathbf P(E)\dashrightarrow Z$ such that 
$Z$ is a smooth projective variety.

We show that
$(q,\alpha) \colon \mathbf{P}(E) \dashrightarrow X\times Z$
is dominant, as in Step~3 of
Example~\ref{ex:torus-blowup-no-local-spray}.
Let $G\subset X\times Z$ be the closure of the image of $(q,\alpha)$.
It is an irreducible compact analytic subset surjecting onto both factors.
Suppose that $G\neq X\times Z$.
Then the general fibres of $G\to Z$ are proper analytic subsets of $X$.
Since $\mathbb T$ contains no curves, every curve in $X$ is contained in $\Gamma$;
indeed, the image under $\pi$ of an irreducible curve in $\widetilde{\mathbb T}$ not contained in the exceptional set would be a curve in $\mathbb T$.
If the general fibres are finite, then $G\to Z$ is generically finite,
which makes $G$ and its image $X$ Moishezon; this contradicts
$a(X)=a(\widetilde{\mathbb T})=a(\mathbb T)=0$.
Otherwise, the general fibres are contained in $\Gamma$, and hence $G\subset\Gamma\times Z$,
which contradicts the surjectivity of $G\to X$.
Thus $(q,\alpha)$ is dominant.
Hence its differential is surjective at a general point, and therefore
\begin{equation}\label{eq:kummer-dominance}
dq_{[\ell]}\bigl(\ker d\alpha_{[\ell]}\bigr)
=T_{q([\ell])}X
\end{equation}
at a general point $[\ell]\in\mathbf P(E)$.

\smallskip
\noindent\emph{Step 4. The rank bound from the Taylor coefficients.}
Since $\delta$ is the square of the displacement,
its degree-two term is
\[
\delta_2(v)=
\left(
\left.\frac{d}{dt}\right|_{t=0}
\bigl(\widetilde s_x(tv)-\widetilde s_x(0_x)\bigr)
\right)^{\otimes2} 
\quad (v\in E_x),
\]
where the derivative is computed in the fixed origin chart.
The assumption implies that
$(d/dt)|_{t=0}\widetilde s_x(tv)$ is surjective
on a neighbourhood of $\Gamma_a$.
Hence $\gamma:=\delta_2^{11}$ is a nonzero
meromorphic section of $S^2E^\vee$.

Work on the dense open subset of $\mathbf P(E)$ where
$\alpha$ is a holomorphic submersion, the base point lies
in $X^\circ$, and $\gamma$ is nonzero.
For $k\geq2$ and $1\leq i\leq j\leq2$, the ratio
$(\delta_k^{ij})^2/\gamma^k$ is homogeneous of degree zero
and defines a meromorphic function on $\mathbf P(E)$;
hence it can be written as the pullback under $\alpha$ of a rational function on $Z$.
This shows that 
\[
(d\delta_k^{ij})_\ell(\eta)
=\frac{k\delta_k^{ij}(\ell)}{2\gamma(\ell)}
d\gamma_\ell(\eta) \quad (\eta\in\ker d(\alpha\circ\nu)_\ell).
\]
Thus all $d\delta_k$ vanish on
$\ker d\gamma\cap\ker d(\alpha\circ\nu)$.
Differentiating \eqref{eq:expansion} term by term gives the
same conclusion for $d\delta$ on $U^\circ$.
Therefore, since $d\gamma$ is scalar-valued, we obtain
\begin{equation}\label{eq:kummer-rank-bound}
\operatorname{rank}
\left((d\delta)_\ell\big|_{\ker d(\alpha\circ\nu)_\ell}\right)
\leq1.
\end{equation}

\smallskip
\noindent\emph{Step 5. A hypersurface mapped into $\Gamma_a$.}
Fix $x_0\in\Gamma_a$, choose coordinates $z=(z_1,z_2)$
centred at $x_0$ with $\Gamma_a=\{z_1=0\}$, and trivialize
$E$ with fibre coordinates $w=(w_1,\ldots,w_r)$.
By the surjectivity of $\varphi_{x_0}$, we may choose the frame so that
$\partial s^1/\partial w_1(0,0)\neq0$.
Shrink the product neighbourhood so that its image under $s$
lies in this coordinate chart in $N_a$ and this derivative
is nonzero throughout.

Applying the same implicit-function construction as in Step~5
of the preceding example, 
we obtain a nonempty patch $S\subset s^{-1}(\Gamma_a)$ with
$p(S)\subset X^\circ$ such that $\nu|_S$ is locally biholomorphic.
As in that construction, choose a small simply connected
neighbourhood $V\subset N_a\setminus\Gamma$ and a fixed
fibre polydisc $D_\varepsilon^r$ such that
\[
\mathcal A:=V\times D_\varepsilon^r\subset U,\qquad
s(\mathcal A)\subset N_a,\qquad S\cap\mathcal A\neq\varnothing.
\]
Replace $S$ by a nonempty patch contained in $\mathcal A$.

The point is that
the displacement $\delta$, which is defined near the zero section over $V$,
cannot be defined along $S$. 
Indeed, the lift used to define it does not extend across $S$:
in local coordinates, such an extension would require
a holomorphic square root of $s^1$, which vanishes to first order
along $S$.
To overcome this difficulty, in the next step we extend the displacement on a branched cover.

\smallskip
\noindent\emph{Step 6. Extending the displacement on a branched cover.}
Consider the fibre product
\[
\widehat{\mathcal A}
:=
\{(v,\widetilde y)\in\mathcal A\times\widetilde N_a:
s(v)=\tau(\widetilde y)\},
\qquad
\operatorname{pr}(v,\widetilde y):=v,
\]
which fits into the Cartesian diagram
\[
\xymatrix@C=5em@R=3em{
\widehat{\mathcal A}
\ar[r]^{(v,\widetilde y)\mapsto\widetilde y}
\ar[d]_{\operatorname{pr}}
&
\widetilde N_a
\ar[d]^{\tau}_{\scriptstyle(t,\zeta)\mapsto(t^2,\zeta)}
\\
\mathcal A \ar[r]_{s}
&
N_a
}
\]
In these coordinates, the subvariety $\widehat{\mathcal A}$ is defined by $t^2-s^1(z,w)=0$, 
which shows, since $\partial s^1/\partial w_1\neq0$, that $S$ is a smooth hypersurface.

The projection $\operatorname{pr}$ is a double cover,
ramified precisely along $t=0$, so its ramification locus is
\[
R:=\{(v,\widetilde y)\in\widehat{\mathcal A}:
\widetilde y\in\pi^{-1}(a)\}.
\]
On $R$, we have $s^1(z,w)=0$ and $\zeta=s^2(z,w)$.
Hence each $v\in s^{-1}(\Gamma_a)\cap\mathcal A$ determines
a unique point of $R$, and the resulting inverse is holomorphic.
Consequently,
$\operatorname{pr}|_R\colon
R \to s^{-1}(\Gamma_a)\cap\mathcal A$ is an isomorphism. 

Since $V \subset X^{\circ}$ is simply connected, 
we obtain a lift of the points of $V$ to $a+B_\rho\setminus\{a\}$, and we denote its
coordinate representation by $u\colon V\to B_\rho\setminus\{0\}$.
By the construction of $Q_{a}$ (see Step~2), we have $u(x)^{\otimes2}=Q_a(x)$.
Define
\[
\widehat\delta(v,\widetilde y)
:=
\bigl(u_a(\pi(\widetilde y))-u(p(v))\bigr)^{\otimes2} 
\quad ( (v,\widetilde y)\in\widehat{\mathcal A}).
\]
This map is holomorphic on all of $\widehat{\mathcal A}$, including $R$,
although $\delta$ is not defined along $S$. 

Near the lift of the zero section on which
$u_a(\pi(\widetilde y))=u(x)$ at $v=0_x$, the target lift
coincides with the fibrewise lift from Step~1 with initial coordinate $u(x)$.
Hence, by the definition of $\delta$, we have
$\widehat\delta=\delta\circ\operatorname{pr}$ on a neighbourhood
of this lifted zero section.
We use this equality only on this initial sheet.

On $R$, the lifted target satisfies $\pi(\widetilde y)=a$,
so its coordinate vanishes. Hence
\begin{equation}\label{eq:kummer-displacement-on-ramification}
\widehat\delta(v,\widetilde y)
=\bigl(-u(p(v))\bigr)^{\otimes2}
=Q_a(p(v)),
\qquad (v,\widetilde y)\in R.
\end{equation}
Thus the displacement on $R$ depends only on the base point $p(v)$.

\smallskip
\noindent\emph{Step 7. Extending the rank bound by analytic continuation.}

Removing the inverse image of the zero section, we put
$\widehat{\mathcal A}^{\times}
:=\widehat{\mathcal A}\setminus\operatorname{pr}^{-1}(0_X)$.
Note that $\widehat{\mathcal A}^{\times}$ is still connected.
On a nonempty open subset of the initial sheet from Step~6,
the map $\operatorname{pr}$ is locally biholomorphic and
$\widehat\delta=\delta\circ\operatorname{pr}$.
The rank bound \eqref{eq:kummer-rank-bound} gives
\begin{equation}\label{eq:upper}
\operatorname{rank}\left(
d\widehat\delta\big|_{\ker d(\alpha\circ\nu\circ\operatorname{pr})}
\right)\leq1 
\end{equation}
on the locus where $\alpha\circ\nu\circ\operatorname{pr}$ is a holomorphic submersion.

\smallskip
\noindent\emph{Step 8. The differential has rank two on $R$.}
Since $\nu(S)$ is open in $\mathbf P(E)$, we can choose $\ell\in S$
such that $\alpha$ is a holomorphic submersion at
$[\ell]:=\nu(\ell)$ and \eqref{eq:kummer-dominance} holds there.
Let $\widehat\ell\in R$ be the unique point over $\ell$.
The restriction of $\nu\circ\operatorname{pr}$ to $R$ is
locally biholomorphic at $\widehat\ell$
and identifies
$T_{\widehat\ell}R\cap
\ker d(\alpha\circ\nu\circ\operatorname{pr})_{\widehat\ell}$
with $\ker d\alpha_{[\ell]}$.
Since $p=q\circ\nu$, we obtain
\[
d(p\circ\operatorname{pr})_{\widehat\ell}
\left(
T_{\widehat\ell}R
\cap\ker d(\alpha\circ\nu\circ\operatorname{pr})_{\widehat\ell}
\right)
=T_{p(\ell)}X.
\]
On $R$, \eqref{eq:kummer-displacement-on-ramification} gives
$\widehat\delta=Q_a\circ p\circ\operatorname{pr}$.
Since $p(\ell)\in V\subset N_a\setminus\Gamma_a$ and
$(dQ_a)_{p(\ell)}$ has rank two, the preceding surjectivity yields
\[
\operatorname{rank}\left(
d\widehat\delta_{\widehat\ell}
\big|_{T_{\widehat\ell}R
\cap\ker d(\alpha\circ\nu\circ\operatorname{pr})_{\widehat\ell}}
\right)=2.
\]
This contradicts \eqref{eq:upper}.
Thus $\varphi$ fails to be surjective at some point of $\Gamma_a$.
Since $a\in\mathbb T[2]$ was arbitrary, this holds for every
exceptional curve. In particular, $X$ admits no local dominating spray.
\end{proof}

\begin{remark}[Oka K3 surfaces that are not elliptic, and non-openness of ellipticity]
\label{rem:kummer-consequences}
Example~\ref{ex:kummer-no-local-spray}, combined with the recent work of
Xie and Zhao \cite{XieZhao2026}, has the following consequences.
Note that a two-dimensional complex torus $\mathbb T$ has algebraic
dimension zero if its N\'eron--Severi group vanishes.
Indeed, if $a(\mathbb T)\ge1$, then a nonconstant meromorphic function
on $\mathbb T$ has a nonzero polar divisor, whose class in
$\mathrm{NS}(\mathbb T)$ is nonzero since $\mathbb T$ is K\"ahler.

\smallskip
\noindent(1) \emph{Oka K3 surfaces that are not elliptic.}
By \cite[Theorem~E\,(2)]{XieZhao2026}, there exist Oka Kummer surfaces
whose N\'eron--Severi lattice coincides with the Kummer lattice.
Since the Picard number of the Kummer surface of $\mathbb T$ equals
$\rho(\mathbb T)+16$, the underlying torus $\mathbb T$ of such a surface
has Picard number zero. Since $H^2(\mathbb T,\mathbb Z)$ is torsion-free, $\mathrm{NS}(\mathbb T)=0$, and hence $\mathbb T$ has algebraic dimension zero.
By Example~\ref{ex:kummer-no-local-spray}, these Oka K3 surfaces are
not elliptic.

\smallskip
\noindent(2) \emph{Ellipticity is not open in families.}
The Fermat quartic $X_0\subset\PP^3$ is a Kummer surface
by a theorem of Mizukami (see \cite[Appendix, Theorem~A.1]{ISZ2011})
and it is Oka by \cite[Corollary~4.3]{XieZhao2026}.
Since $X_0$ is projective, it is elliptic by \cite{Forstneric2026MRL}
(or by Theorem~\ref{thm-main}).
Let $\mathbb T_0$ be a two-dimensional complex torus whose Kummer
surface is $X_0$.
Complex tori $\mathbb C^2/\Lambda$ are parametrized by an open subset of the space of lattices, and
those of positive algebraic dimension, that is, those containing a curve, form a countable union of proper analytic subsets
of the parameter space.
Hence there is a holomorphic family $\{\mathbb T_t\}_{t\in\mathbb D}$
of two-dimensional complex tori over the unit disc $\mathbb D$,
with the given central fibre $\mathbb T_0$, such that $\mathbb T_t$
has algebraic dimension zero for every $t$ outside a countable subset of $\mathbb D$.
Performing the Kummer construction in this family
(blowing up the sixteen sections of two-torsion points and taking the quotient
by the fibrewise involution), we obtain a holomorphic family
$\{X_t\}_{t\in\mathbb D}$ of Kummer surfaces, that is, a proper holomorphic
submersion with fibres $X_t$.
Its central fibre $X_0$ is elliptic, whereas $X_t$ is not elliptic for every
$t$ outside a countable subset of $\mathbb D$ by
Example~\ref{ex:kummer-no-local-spray}.
Therefore ellipticity is not an open condition in holomorphic families
of compact complex manifolds.
\end{remark}

\begin{remark}[Ellipticity is neither open nor closed in families of Oka manifolds]
\label{rem:torus-blowup-families}
Example~\ref{ex:torus-blowup-no-local-spray} shows that ellipticity
is neither an open nor a closed condition in holomorphic families of
compact Oka manifolds.
For comparison, Forstneri\v{c} and L\'arusson showed that the Oka property
itself is not closed in families \cite{ForstnericLarusson2014IMRN};
in our families, every fibre is Oka.

There is a holomorphic family $\{\mathbb T_b\}_{b\in B}$ of
two-dimensional complex tori over a connected complex manifold $B$
such that both the fibres that are projective and the fibres with
vanishing N\'eron--Severi group occur over dense subsets of $B$.
Indeed, take the family of all tori $\mathbb C^2/(I_2,b)\,\mathbb Z^4$
with $\det(\operatorname{Im}b)>0$; its period map identifies $B$ with an
open subset of the period domain
$Q=\{[\sigma]:\sigma^2=0,\ \sigma\cdot\bar\sigma>0\}$
in $\mathbf P(H^2(\mathbb T,\mathbb C))$, where $\mathbb T$ is a fixed fibre
(see the proof of \cite[Theorem~II]{Shioda1978}), and
$\mathrm{NS}(\mathbb T_b)$ consists of the integral classes orthogonal to
$\sigma_b$.
For each nonzero class $c$, the locus $\{c\cdot\sigma=0\}$ is a proper
hypersurface of $Q$, so $\mathrm{NS}(\mathbb T_b)=0$ off countably many
hypersurfaces.
On the other hand, every $\sigma\in Q$ is orthogonal to some real class
$c$ with $c^2>0$, since the orthogonal complement of the positive plane
$\langle\operatorname{Re}\sigma,\operatorname{Im}\sigma\rangle$ has
signature $(1,3)$; approximating $c$ by rational classes shows that the
loci $\{c\cdot\sigma=0\}$ with $c$ rational and $c^2>0$ are dense in $Q$,
and over them $\mathbb T_b$ is projective \cite[Theorem~IV.6.2]{BHPV2004}.

Blowing up the zero section, we obtain a holomorphic family
$\{X_b\}_{b\in B}$, where $X_b$ is the blow-up of $\mathbb T_b$ at a point.
Every $X_b$ is Oka by \cite[Corollary~6.4.13]{Forstneric2017book}.
If $\mathbb T_b$ is projective, then so is $X_b$, and hence $X_b$ is
elliptic by \cite{Forstneric2026MRL} (or by Theorem~\ref{thm-main}).
If $\mathrm{NS}(\mathbb T_b)=0$, then $\mathbb T_b$ has algebraic dimension
zero (Remark~\ref{rem:kummer-consequences}), and hence $X_b$ is not elliptic
by Example~\ref{ex:torus-blowup-no-local-spray}.
Thus both the elliptic and the non-elliptic fibres occur over dense
subsets of $B$, so ellipticity is neither open nor closed in this family.
\end{remark}

\section{Globalization of local sprays}\label{sec:globalization}

Let $f\colon U\to Y$ be a holomorphic map between complex manifolds. A \emph{product spray} with \emph{core} $f$ is a holomorphic map $H\colon U\times P\to Y$ such that $H(\cdot,0)=f$, where $P$ is a ball centred at $0$ in some $\CC^M$; thus a product spray is a local spray over $f$ on the trivial bundle $U\times\CC^M$ that is defined on $U\times P$. It is said to be \emph{dominating} over $S\subset U$ if $\partial_\zeta H(x,\zeta)$ maps onto $T_{Y,H(x,\zeta)}$ for all $x\in S$ and $\zeta\in P$. Such maps are called sprays of maps in \cite[Chap.~5]{Forstneric2017book}.

\begin{theorem}\label{thm:oka}
Let $X$ be a weakly pseudoconvex manifold with a K\"ahler form $\omega_X$, and let $E\to X$ be a vector bundle of rank $N\ge1$ with a Hermitian metric.
Assume that, for some smooth Hermitian metric on $K_X^{-1}$ and some constant $c\ge0$,
\begin{equation}\label{eq:hyp}
\begin{gathered}
\langle \sqrt{-1}\Theta_E(v,\bar v)\xi,\xi\rangle\le-\omega_X(v,\bar v)\,|\xi|^2\qquad(v\in T_{X,x},\ \xi\in E_x,\ x\in X),\\
\sqrt{-1}\Theta(K_X^{-1}\otimes\det E)\ge-c\,\omega_X.
\end{gathered}
\end{equation}
We identify $E$ with its total space and let $A$ be its zero section.
Let $Y$ be an Oka manifold, and let $f\colon E\to Y$ be a continuous map that is holomorphic on a neighbourhood $U$ of $A$ and admits a product spray with core $f|_U$ that is dominating over $U\setminus A$.
Then, for every $r\ge1$, there exists a holomorphic map $\tilde f\colon E\to Y$ that agrees with $f$ to order $r$ along $A$.
\end{theorem}

Theorem~\ref{thm:local-to-global-spray} is deduced from Theorem~\ref{thm:oka} in Section~\ref{sec:proofmain}.

\emph{Strategy of the proof of Theorem~\ref{thm:oka}.} We follow the proof of the Oka principle for Stein sources \cite[Chap.~5]{Forstneric2017book}. We exhaust the total space $E$ by the disc bundles $\{\theta_w<1\}$ defined below, and we pass from one disc bundle to a larger one in finitely many steps, each of which attaches a convex bump $V$ to a set $U\supset A$ (Lemma~\ref{lem:chain}). At each step, the current map $F$ is holomorphic on $U$, and on the convex set $V$ the Oka property of $Y$ provides a holomorphic map $G$ that approximates $F$ on $U\cap V$. To glue $F$ and $G$, we use dominating product sprays $H$ with core $F$ and $T$ with core $G$. Near $U\cap V$, the old spray can be written as $H=T\circ\gamma$, where $\gamma(x,\cdot)$ is close to the identity map of the parameter space (Lemma~\ref{lem:transition}). If $\gamma=\beta\circ\alpha^{-1}$ with $\alpha$ defined over $U$ and $\beta$ over $V$, then $H\circ\alpha$ and $T\circ\beta$ agree on $U\cap V$ and together form a product spray over $U\cup V$ (Lemma~\ref{lem:gluesprays}). This compositional splitting is a nonlinear version of the additive Cousin problem $c=b-a$, from which it is obtained by iteration (Lemma~\ref{lem:nonlin}); the additive problem is solved via the $\dbar$-equation (Lemma~\ref{lem:linsplit}). Here $U$ contains $A$ and need not be compact or Stein, so we solve the $\dbar$-equation with $L^2$ estimates for a weight $q^{-m}$ that is singular along $A$; this weight also forces $\alpha$ to agree with the identity to order $r$ along $A$, which yields the interpolation along $A$. Since the disc bundles are not compact, we also need bounds near the ends of $X$ (Lemmas~\ref{lem:pointwise} and~\ref{lem:chi}). Finally, the functions $\theta_w$ have no critical points off $A$ (Lemma~\ref{lem:total}(d)); hence, unlike in the Stein case, no handles need to be attached.

\subsection{A splitting lemma on weakly pseudoconvex manifolds}\label{subsec:heftung}

In this subsection and in Subsection~\ref{sec:total}, let $X$, $\omega_X$, $E$, $c$ and the metric on $K_X^{-1}$ be as in Theorem~\ref{thm:oka}, and let $\rho\ge0$ be a smooth psh exhaustion function of $X$. We identify $E$ with its total space, denote by $\pi\colon E\to X$ the projection and by $A$ the zero section, and set $n_E=\dim E$. Put
\[
u(\xi)=|\xi|^2,\qquad q=u\,e^{u},\qquad\Psi=\rho\circ\pi+u,\qquad\theta_w=q\,e^{w\circ\pi}\ \ (w\in C^\infty(X)),
\]
and let $\cE$ be the Euler vector field of $E$. We say that a holomorphic function \emph{vanishes to order $r$ along $A$} if all its derivatives of order $<r$ vanish on $A$.

\subsubsection*{The total space and $L^2$ estimates}

\begin{lemma}\label{lem:total}
\begin{enumerate}[label=(\alph*)]
\item $\Psi$ is a psh exhaustion function of $E$, and $\omega_E=\pi^*\omega_X+\ddc u$ is a K\"ahler form.
\item $\ddc\log q\ge\omega_E$ on $E\setminus A$. Consequently, for a psh function $w$ on $X$, $\log\theta_w$ is psh on $E$, $\theta_w$ is strictly psh on $E\setminus A$, and $-\log(b-\theta_w)+\Psi$ is a psh exhaustion function of $\{\theta_w<b\}$.
\item With the induced metric $h_0$ on $K_E^{-1}\cong\pi^*(K_X^{-1}\otimes\det E)$ and $m\ge c+1$,
\begin{equation}\label{eq:positive}
\sqrt{-1}\Theta(K_E^{-1},h_0)+m\,\ddc\log q\ge\omega_E\qquad\text{on }E\setminus A .
\end{equation}
\item $d\theta_w(\cE+\bar\cE)=2(1+u)\,\theta_w$; in particular, $\theta_w$ has no critical points on $E\setminus A$.
\item In local frames, $q^{1/2}$ is locally Lipschitz and $u$ is comparable to the squared Euclidean norm $|\xi|^2$.
\end{enumerate}
\end{lemma}

\begin{proof}
At $\xi_0\in E_{y_0}$, in a frame normal at $y_0$ and for a tangent vector $(v,\tau)$, \eqref{eq:hyp} gives
\[
\begin{aligned}
\ddc u(v,\tau)&\ge|\tau|^2+\omega_X(v,\bar v)|\xi_0|^2,\\
\ddc\log u(v,\tau)&\ge\frac{|\tau|^2|\xi_0|^2-|\langle\tau,\xi_0\rangle|^2}{|\xi_0|^4}+\omega_X(v,\bar v)\ge\omega_X(v,\bar v)
\end{aligned}
\]
for $\xi_0\ne0$. This gives (a) and $\ddc\log q=\ddc\log u+\ddc u\ge\omega_E$; the rest of (b) follows from $\ddc e^\phi\ge e^\phi\ddc\phi$ and the compactness of $\{\Psi\le M+\log b,\ \theta_w\le b-e^{-M}\}$. For (c), we use the exact sequence $0\to\pi^*E\to T_E\to\pi^*T_X\to0$, \eqref{eq:hyp} and $\pi^*\omega_X\le\omega_E$. For (d), the identities $\partial u(\cE)=u$ and $\partial(w\circ\pi)(\cE)=0$ give $\partial\log\theta_w(\cE)=1+u$. Finally, (e) is clear from $u=|h(y)^{1/2}\xi|^2$.
\end{proof}

For $m\ge c+1$, let $d\mu$ be the volume form of $h_0$ and put $d\mu_m=q^{-m}d\mu$. For a $\C^k$-valued $(0,j)$-form $g$ on an open set $S\subset E$, identified with a $K_E^{-1}$-valued $(n_E,j)$-form, put $\|g\|_S^2=\int_S|g|^2d\mu_m$, where $|g|$ is measured with respect to $\omega_E$.

\begin{lemma}\label{lem:L2}
Let $D\subset E$ be a weakly pseudoconvex open set, let $j\ge1$, and let $g$ be a $\dbar$-closed $(0,j)$-form on $D$ with $L^2_{\mathrm{loc}}$ coefficients and $\|g\|_D<\infty$. Then the solution $v$ of $\dbar v=g$ with minimal norm exists, depends linearly on $g$, and satisfies $\|v\|_D\le\|g\|_D$.
\end{lemma}

\begin{proof}
\emph{$D\setminus A$ is complete K\"ahler.} Let $\Phi\ge0$ be a smooth psh exhaustion function of $D$. Then $\widehat\omega_D=\omega_E+\ddc(\Phi^2)$ is a complete K\"ahler form on $D$: indeed, $\ddc(\Phi^2)\ge2\sqrt{-1}\partial\Phi\wedge\dbar\Phi$, so the exhaustion function $\Phi$ has bounded gradient. Let $\eta\colon\R\to\R$ be smooth, convex and increasing, with $\eta(t)=-\log(-t)$ for $t\le-1$. Since $\log q$ is psh (Lemma~\ref{lem:total}(b)), the form $\widehat\omega=\widehat\omega_D+\ddc\,\eta(\log q)$ is K\"ahler on $D\setminus A$. On the set where $\log q\le-1$, we have $\ddc\,\eta(\log q)\ge\sqrt{-1}\partial\tau\wedge\dbar\tau$ with $\tau=\log(-\log q)$, and $\tau\to\infty$ at $A$; hence $\widehat\omega$ is complete on $D\setminus A$.

\emph{Solution.} On $D\setminus A$, the metric $h_0q^{-m}$ of $K_E^{-1}$ has curvature $\ge\omega_E$ by \eqref{eq:positive}, so the curvature operator on $K_E^{-1}$-valued $(n_E,j)$-forms is $\ge j\ge1$ with respect to $\omega_E$. Demailly's $L^2$ existence theorem \cite[Th\'eor\`eme~4.1]{Demailly1982}, which applies on the complete K\"ahler manifold $D\setminus A$ with estimates for the (not necessarily complete) K\"ahler form $\omega_E$, gives $v$ on $D\setminus A$ such that $\dbar v=g$ and $\|v\|_{D\setminus A}\le\|g\|_D$. Since $q$ is bounded on compact sets, $v$ has $L^2_{\mathrm{loc}}$ coefficients on $D$, and the equation $\dbar v=g$ extends across the analytic set $A$ \cite[Lemme~6.9]{Demailly1982}. The minimal solution is the orthogonal projection of any solution onto the orthogonal complement of $\ker\dbar$.
\end{proof}

\begin{lemma}\label{lem:mult}
If $m\ge N+r-1$ and $a$ is holomorphic on an open set $S\subset E$ and satisfies $\int_{S'}|a|^2d\mu_m<\infty$ for all $S'\Subset S$, then $a$ vanishes to order $r$ along $A\cap S$.
\end{lemma}

\begin{proof}
Locally, $d\mu_m$ is comparable to $|t|^{-2m}dV$, where $t$ is the fibre coordinate. In the expansion $a=\sum_\beta a_\beta(x)t^\beta$, the monomials are orthogonal, and $\int|t^\beta|^2|t|^{-2m}dV(t)=\infty$ if $|\beta|\le r-1$.
\end{proof}

Fix a continuous function $C_E\colon E\to[1,\infty)$ such that the $\omega_E$-ball $B(x,4/C_E(x))$ lies in a chart in which $\omega_E$, the Euclidean metric and $d\mu/dV$ are comparable with constant $C_E(x)$, and such that $q^{1/2}$ is $C_E(x)$-Lipschitz on $B(x,1/C_E(x))$.

\begin{lemma}\label{lem:pointwise}
There exists a continuous function $\kappa\colon E\to[1,\infty)$ with the following property.
Let $R\colon E\to(0,1]$ be a $C^1$ function, put $W_R=\{q<R^2\}$, and let $0<\delta\le\frac12$, $\xi\in W_{(1-\delta)R}$ and $\Lambda\ge2C_E(\xi)$. If $|d\log R|\le\Lambda$ on $B(\xi,2/\Lambda)$ and $f$ is holomorphic on $W_R\cap B(\xi,2/\Lambda)$ with norm $\le b$ there, then
\[
|f(\xi)|\le\kappa(\xi)\bigl(C_E(\xi)+\Lambda\bigr)^{n_E}\delta^{-n_E}R(\xi)^{m-n_E}\,b .
\]
\end{lemma}

\begin{proof}
Put $t=\frac\delta4\min\{R(\xi)/C_E(\xi),1/\Lambda\}$. For $y\in B(\xi,t)$, we have $q(y)^{1/2}<(1-\frac34\delta)R(\xi)$ and $R(y)\ge(1-\frac\delta4)R(\xi)$, so $B(\xi,t)\subset W_R\cap B(\xi,2/\Lambda)$ and $q\le R(\xi)^2$ there.
The mean value inequality on a Euclidean ball of radius $t/C_E(\xi)$ gives $|f(\xi)|\le C(\xi)t^{-n_E}R(\xi)^mb$, and $t^{-1}\le4\delta^{-1}(C_E(\xi)+\Lambda)R(\xi)^{-1}$.
\end{proof}

\subsubsection*{Splitting}

\begin{lemma}[Additive splitting]\label{lem:linsplit}
Let $m\ge\max\{c+1,N+r-1\}$, and let $D=D_A\cup D_B\subset E$ be weakly pseudoconvex, where $D_A,D_B$ are open, $\overline{D_B}$ is compact and disjoint from $A$, and $D_C=D_A\cap D_B$. Let $\chi\in C^\infty(E,[0,1])$ satisfy $\chi=0$ near $D_A\setminus D_B$ and $\chi=1$ near $D_B\setminus D_A$, and assume that $\supp d\chi$ is compact, disjoint from $A$, and satisfies $\supp d\chi\cap D\subset D_C$. For a bounded holomorphic map $c\colon D_C\to\C^k$, let $v_c$ be the minimal solution of $\dbar v=c\,\dbar\chi$ on $D$, and put $a_c=v_c-\chi c$ on $D_A$ and $b_c=v_c+(1-\chi)c$ on $D_B$. Then $a_c$ and $b_c$ are holomorphic, $b_c-a_c=c$, $a_c$ vanishes to order $r$ along $A$, and \[
\|a_c\|_{D_A}+\|b_c\|_{D_B}\le C_0\sup|c|,\qquad C_0=2\,\mu_m(D_B)^{1/2}+2\,\|\dbar\chi\|_{D},
\]
where $\mu_m(D_B)=\int_{D_B}d\mu_m<\infty$. Moreover, holomorphic dependence on parameters is preserved.
\end{lemma}

\begin{proof}
The form $c\,\dbar\chi$ on $D$, extended by $0$, is smooth and $\dbar$-closed, since $\supp d\chi\cap D\subset D_C$. Its norm is at most $\sup|c|\,\|\dbar\chi\|_D$, which is finite because $\supp d\chi$ is compact and disjoint from $A$, so that $q$ is bounded below on it.
Similarly, $\chi c$ and $(1-\chi)c$ extend from $D_C$ to $D_A$ and to $D_B$, respectively, and are supported in $D_B$, so their norms are at most $\sup|c|\,\mu_m(D_B)^{1/2}$; here $\mu_m(D_B)<\infty$ because $\overline{D_B}$ is compact and disjoint from $A$. Lemma~\ref{lem:L2} gives $\|v_c\|_D\le\|c\,\dbar\chi\|_D$, and the estimate follows.
Near $A$, we have $a_c=v_c$, so $a_c$ vanishes to order $r$ along $A$ by Lemma~\ref{lem:mult}. The minimal solution depends linearly and boundedly on the data, which gives holomorphic dependence on parameters.
\end{proof}

The next lemma is the analogue of the splitting lemma \cite[Proposition~5.8.1]{Forstneric2017book} (a generalization of Cartan's lemma on invertible matrices) and of \cite[Theorem~5.2]{Stopar2013IJM}. It is purely formal: the geometry enters only through the hypothesis on additive splittings, which is verified for the pairs $(U_\lambda,V_\lambda)$ in Lemma~\ref{lem:radialsplit}. Since the weighted $L^2$ estimates give supremum bounds only on slightly smaller sets, we replace the implicit function theorem on a fixed Banach space used in \cite{Forstneric2017book,Stopar2013IJM} by a Newton iteration on shrinking domains.

\begin{lemma}[Splitting lemma]\label{lem:nonlin}
Let $U_j,V_j$ ($j\ge0$) be decreasing sequences of open sets with $V_0\cap A=\varnothing$, put $C_j=U_j\cap V_j$, and let $U_\infty\subset\bigcap U_j$ and $V_\infty\subset\bigcap V_j$ be open.
Let $\rho_j=\rho_\infty+(\rho_0-\rho_\infty)2^{-j}$, and let $B_j\subset\C^k$ be the ball of radius $\rho_j$ centred at $0$. Assume that every bounded holomorphic map $c$ on $C_j\times B_j$ splits as $c=b-a$ on $C_{j+1}\times B_j$, where $a$ is defined on $U_{j+1}\times B_j$ and $b$ on $V_{j+1}\times B_j$, $\sup|a|+\sup|b|\le C2^{pj}\sup|c|$ ($C\ge1$, $p\ge0$), and $a(\cdot,\zeta)$ vanishes to order $r$ along $A$. If $\gamma$ is holomorphic on $C_0\times B_0$ and $\epsilon_0=\sup|\gamma-\zeta|\le\epsilon_*:=(\rho_0-\rho_\infty)/(12C2^{p+1})$, then there exist holomorphic maps $\alpha\colon U_\infty\times B_\infty\to B_0$ and $\beta\colon V_\infty\times B_\infty\to\C^k$ such that $\gamma(x,\alpha(x,\zeta))=\beta(x,\zeta)$ on $(U_\infty\cap V_\infty)\times B_\infty$, $\sup|\alpha-\zeta|+\sup|\beta-\zeta|\le C'\epsilon_0$ with $C'=C'(C,p,\rho_0,\rho_\infty)$, and $\alpha(\cdot,\zeta)-\zeta$ vanishes to order $r$ along $A$.
\end{lemma}

\begin{proof}
Put $s_j=\rho_j-\rho_{j+1}$ and $\gamma_0=\gamma$. Given $\gamma_j$ on $C_j\times B_j$ with $\epsilon_j=\sup|\gamma_j-\zeta|$, we split $\gamma_j-\zeta=b_j-a_j$, put $\eta_j=C2^{pj}\epsilon_j$, $\alpha_j=\zeta+a_j$ and $\beta_j=\zeta+b_j$, and assume that $\eta_j\le s_j/8$. By the Cauchy estimates, $|\partial_\zeta a_j|,|\partial_\zeta b_j|\le\frac14$ on the ball of radius $\rho_{j+1}+s_j/2$; hence $\beta_j^{-1}$ exists on the ball of radius $\rho_{j+1}+s_j/4$ by the contraction principle, and $\gamma_{j+1}=\beta_j^{-1}\circ\gamma_j\circ\alpha_j$ is defined on $C_{j+1}\times B_{j+1}$. Writing $\gamma_j\circ\alpha_j=\zeta+b_j+Q$ with $|Q|\le2\epsilon_j\eta_j/s_j$, we obtain $\epsilon_{j+1}\le\frac83\epsilon_j\eta_j/s_j\le K2^{(p+1)j}\epsilon_j^2$, where $K=6C/(\rho_0-\rho_\infty)$. Hence $d_j=K2^{(p+1)(j+1)}\epsilon_j$ satisfies $d_{j+1}\le d_j^2$ and $d_0\le\frac12$, so $d_j\le2^{-2^j}$ and $\eta_j\le s_j/12$, which keeps the induction going.
Since $\sum\eta_i/s_i<\infty$, the compositions $\alpha_0\circ\dots\circ\alpha_j$ and $\beta_0\circ\dots\circ\beta_j$ have uniformly bounded Lipschitz constants, and they converge to maps $\alpha,\beta$ with the stated properties, since $\gamma\circ\alpha_0\circ\dots\circ\alpha_j=\beta_0\circ\dots\circ\beta_j\circ\gamma_{j+1}$ and $\gamma_j\to\zeta$. By the chain rule, vanishing to order $r$ along $A$ is preserved under composition with the $\alpha_i$ and under locally uniform limits.
\end{proof}

\subsection{Jet extension theorem from negative bundles to Oka manifolds}\label{sec:total}

We construct the product sprays needed in Theorem~\ref{thm:oka}.

\begin{lemma}\label{lem:vanishing}
Let $V\to U$ be a vector bundle with a section $\sigma$, put $A=\sigma^{-1}(0)$, and let $s_0$ be a dominating local spray over $f\colon U\to Y$ on a bundle $P_0\to U$. Let $\tau_1,\dots,\tau_M$ be sections of $\Hom(V,P_0)$ spanning it at every point, and let $P$ be the unit ball of $\C^M$. Then, for some neighbourhood $U'$ of $A$,
\[
H(x,\zeta)=s_0\Bigl(\sum_k\zeta_k\,\tau_k(x)\bigl(\sigma(x)\bigr)\Bigr),\qquad x\in U',\ \zeta\in P,
\]
is a product spray with core $f|_{U'}$ that is dominating over $U'\setminus A$.
\end{lemma}

\begin{proof}
The points $\eta$ at which the fibre derivative of $s_0$ is surjective form a neighbourhood $\Omega'$ of the zero section.
Choose metrics on $V$ and $P_0$ and a continuous function $\delta>0$ such that $\{|\eta|<\delta(p(\eta))\}\subset\Omega'$, and let $U'$ be the set where $(\sum_k|\tau_k|^2)^{1/2}|\sigma|<\delta$. For $x\in U'$ and $\zeta\in P$, the point $\psi_x(\zeta)=\sum_k\zeta_k\tau_k(x)(\sigma(x))$ lies in $\Omega'$ by the Cauchy--Schwarz inequality, and $\partial_\zeta H=d(s_0|_{P_{0,x}})_{\psi_x(\zeta)}\circ\psi_x$. If $\sigma(x)\ne0$, then $\tau\mapsto\tau(\sigma(x))$ maps $\Hom(V_x,P_{0,x})$ onto $P_{0,x}$, so $\psi_x$ is surjective.
\end{proof}

\begin{lemma}\label{lem:fibresubmersive}
Let $L\to X$ be a Hermitian line bundle, let $E=L\otimes\C^N$ with projection $\pi$ and zero section $A$, and let $g$ be a holomorphic map from a neighbourhood $W$ of $A$ to $Y$ such that $d(g|_{E_y})_{0_y}$ is surjective for all $y$. Let $P$ be the unit Hilbert--Schmidt ball of $\C^{N\times N}$. Then, for some neighbourhood $U\subset W$ of $A$, the map $H(\xi,\Lambda)=g(\xi+\Lambda\xi)$ on $U\times P$ is a product spray with core $g|_U$ that is dominating over $U\setminus A$.
\end{lemma}

\begin{proof}
We apply Lemma~\ref{lem:vanishing} on $W$ with $V=P_0=\pi^*E$, the tautological section $\sigma(\xi)=\xi$, the local spray $s_0(\xi,\eta)=g(\xi+\eta)$ over $g$, whose fibre derivative $d(g|_{E_{\pi(\xi)}})_\xi$ is surjective after shrinking $W$, and the matrix units $\tau_{lj}=e_le_j^T$, which span $\pi^*\End E\cong W\times\C^{N\times N}$ and satisfy $\sum\Lambda_{lj}\tau_{lj}(\xi)=\Lambda\xi$.
\end{proof}

\subsubsection*{Disc bundles and convex bumps}
Fix $r\ge2$ such that $m=n_E+r-1\ge c+1$ (so that $m\ge N+r-1$), fix an open neighbourhood $U$ of $A$, and put $L_j=\{\Psi\le j\}$.

\begin{lemma}\label{lem:chi}
There exists a smooth convex increasing function $\chi\colon[0,\infty)\to[5,\infty)$ with $\chi'\to\infty$ such that
\begin{enumerate}[label=(\alph*)]
\item $\theta_{\chi\circ\rho}\ge2$ on $E\setminus U$;
\item if $\alpha\in[\frac12,2]$, $R=\alpha e^{-\chi(\rho\circ\pi)/2}$, $\rho(\pi(x))\ge1$ and $q(x)\le1$, then $R<1$, and there exists $\Lambda\ge2C_E(x)$ such that $|d\log R|\le\Lambda$ and $\rho\circ\pi\ge\rho(\pi(x))-1$ on $B(x,2/\Lambda)$ and $\kappa(x)(C_E(x)+\Lambda)^{n_E}R(x)\le1$.
\end{enumerate}
\end{lemma}

\begin{proof}
Let $Z_t=\{\rho\circ\pi\le t+2,\ q\le4\}$, which is compact.
Choose a continuous increasing function $C(t)\ge1$ and a continuous decreasing function $\mu(t)\in(0,1]$ such that $|d(\rho\circ\pi)|,C_E,\kappa\le C(t)$ on $Z_t$ and $q\ge\mu(t)$ on $\{\rho\circ\pi\le t\}\setminus U$. Put $b=\max\{5,\log2-\log\mu\}$, $\Gamma=16C^3$, $\nu=e^{-b/(2n_E)}$, $s(t)=\min\{(2n_E\Gamma(t+1))^{-1},2^{-t-1}\nu(t+1)\}$, choose a smooth strictly convex function $p$ with $p\ge-\log s$, $p'\ge1$ and $p(t)\ge t^2$, and put $\tau=\int_t^\infty e^{-p}$ and $\chi=-2n_E\log\tau$. By strict convexity, $\tau<e^{-p}/p'$; hence $\chi'=2n_Ee^{-p}e^{\chi/(2n_E)}\in[2n_Ep',\,e^{\chi/(2n_E)}/\Gamma]$, so $\chi'\to\infty$. Moreover, $0<\chi''\le(\chi')^2/(2n_E)$, and $\chi\ge b$ since $\tau(t)\le\sum_{j\ge\lfloor t\rfloor}s(j)\le\nu(t)$.

(a) Off $U$, we have $q\ge\mu\ge2e^{-b}\ge2e^{-\chi}$.
(b) Let $t_x=\rho(\pi(x))$, $C=C(t_x)$ and $\Lambda=4C^2\chi'(t_x)\ge2C_E(x)$. Along curves from $x$ of length $<2/\Lambda\le1/(2C\chi'(t_x))$ that stay in $Z_{t_x}$, the function $\rho\circ\pi$ changes by less than $1/(2\chi'(t_x))\le\frac12$ and $q^{1/2}\le\frac32$, so such curves never leave $Z_{t_x}$. Since $(1/\chi')'\ge-\frac12$, we have $\chi'\le2\chi'(t_x)$ on $[t_x-1,t_x+1/\chi'(t_x)]$, so $|d\log R|\le C\chi'(t_x)\le\Lambda$ on $B(x,2/\Lambda)$. Finally, $\kappa(C_E+\Lambda)^{n_E}R\le C(8C^2\chi')^{n_E}2e^{-\chi/2}\le2C(8C^2/\Gamma)^{n_E}\le1$.
\end{proof}

For $T\ge1$, we consider functions $r$ on $E$ satisfying
\begin{equation}\label{eq:ends}
r=\theta_{\chi\circ\rho}-1\qquad\text{on }\pi^{-1}(\{\rho\ge T\}).
\end{equation}

\begin{definition}\label{def:bump}
A \emph{convex bump} consists of smooth psh functions $r_D\le r_U$ on $E$ satisfying \eqref{eq:ends}, with $r_D\ge-2$, $r_U=-1$ on $A$, and $\{r_D\le1\}\cap\pi^{-1}(K)$ compact for every compact set $K\subset X$, together with a chart $z\colon\Omega\to\{|z|<3\}$, where $\Omega\subset\pi^{-1}(\{\rho<T+1\})\setminus A$, in which $r_U,r_D$ are strictly convex, such that $r_U=r_D$ outside $\{|z|<1\}$. For $0<\lambda\le\frac14$, put $U_\lambda=\{r_U<\lambda\}$, $D_\lambda=\{r_D<\lambda\}$, $V_\lambda=D_\lambda\cap\{|z|^2<4+\lambda\}$ and $C_\lambda=U_\lambda\cap V_\lambda$, and define $\chi_\Omega=h(|z|^2)$ on $\Omega$ and $\chi_\Omega=0$ on $E\setminus\Omega$, where $h\in C^\infty(\R,[0,1])$ satisfies $h=1$ on $(-\infty,1]$ and $h=0$ on $[2,\infty)$.
\end{definition}

The pair $(U_\lambda,V_\lambda)$ plays the role of a Cartan pair \cite[Definition~5.7.1]{Forstneric2017book} or of a $1$-convex Cartan pair \cite[Definition~5.1]{Stopar2013IJM}. Of the defining properties of a Cartan pair, we use only the separation property (Lemma~\ref{lem:pairs}(c)), which gives the cut-off function in the additive splitting, and the convexity of $V_\lambda$ and $C_\lambda$ (Lemma~\ref{lem:pairs}(a)), which is needed for approximation and for transition maps.
The Stein property of $D_\lambda$, which in \cite[Lemma~5.8.2]{Forstneric2017book} yields solutions of the $\dbar$-equation with bounds, is replaced by the weak pseudoconvexity of $D_\lambda$ (Lemma~\ref{lem:pairs}(b)) and the weighted $L^2$ estimates of Lemma~\ref{lem:L2}; as in \cite[Proposition~4.5]{Stopar2013IJM}, the $\dbar$-data are supported away from $A$.

\begin{lemma}\label{lem:pairs}
\begin{enumerate}[label=(\alph*)]
\item $D_\lambda=U_\lambda\cup V_\lambda$; $V_\lambda$ and $C_\lambda=U_\lambda\cap\{|z|^2<4+\lambda\}$ are convex in the chart; $\overline{V_\lambda}\subset\Omega$ is compact.
\item $U_\lambda$ and $D_\lambda$ are weakly pseudoconvex.
\item On $D_\lambda$: $\{\chi_\Omega<1\}\subset U_\lambda$, $\{\chi_\Omega>0\}\subset V_\lambda$, $\supp d\chi_\Omega\subset C_\lambda$.
\item There exists $c_1>0$ such that, for $\lambda'<\lambda$, the sets $V_{\lambda'}$ and $U_{\lambda'}\cap\pi^{-1}(\{\rho\le T+1\})$ have distance $\ge c_1(\lambda-\lambda')$ from $E\setminus V_\lambda$ and $E\setminus U_\lambda$.
\end{enumerate}
\end{lemma}

\begin{proof}
Assertions (a) and (c) follow since $r_D<r_U$ only in $\{|z|<1\}$ and $\supp d\chi_\Omega\subset\{1\le|z|^2\le2\}$. For (b), use $-\log(\lambda-r)+\Psi$ as in Lemma~\ref{lem:total}(b). For (d), bound $|dr_U|,|dr_D|,|d|z|^2|$ on a neighbourhood of the compact set $(\{r_D\le1\}\cap\pi^{-1}(\{\rho\le T+1\}))\cup\{|z|^2\le\frac{17}4\}$.
\end{proof}

\begin{lemma}[Decomposition into convex bumps]\label{lem:chain}
Let $T\ge1$, let $w\ge w'$ be smooth psh functions on $X$ with $w=w'=\chi\circ\rho$ on $\{\rho\ge T\}$, put $r_s=\theta_{(1-s)w+sw'}-1$, and let $0<\epsilon\le\frac14$. Then there exist convex bumps $(r^{(i)}_U,r^{(i)}_D)$, $i=1,\dots,I$, with charts in $\{\theta_w>1-2\epsilon\}$, such that $r_1\le r^{(i)}_D\le r^{(i)}_U\le r_0$ and, for $|\lambda|\le\epsilon/2$,
\[
\{r^{(1)}_U<\lambda\}=\{r_0<\lambda\},\quad\{r^{(i)}_D<\lambda\}=\{r^{(i+1)}_U<\lambda\},\quad\{r^{(I)}_D<\lambda\}=\{r_1<\lambda\}.
\]
\end{lemma}

\begin{proof}
This is the usual construction for noncritical strongly psh exhaustion functions \cite[Lemma~5.10.3]{Forstneric2017book}, carried out over $\{\rho\le T\}$. By Lemma~\ref{lem:total}, the functions $r_s$ are psh on $E$, strictly psh and without critical points on $E\setminus A$, equal to $-1$ on $A$, and independent of $s$ over $\{\rho\ge T\}$. Since $r_s+1=q\,e^{((1-s)w+sw')\circ\pi}$ and $w'\le w$, they are nonincreasing in $s$, and $r_{s'}\to r_s$ in $C^2$ on compact sets as $s'\to s$. (They are not affine in $s$; only $\log(r_s+1)$ is.)

\emph{Charts.} The set $Z$ of pairs $(s,x)\in[0,1]\times E$ with $\rho(\pi(x))\le T$ and $|r_s(x)|\le\epsilon$ is compact and does not meet $[0,1]\times A$. By Narasimhan's lemma, for every $(s_*,x_*)\in Z$ there exist a chart $z\colon\Omega\to\{|z|<3\}$ centred at $x_*$, with $\Omega\subset\{\theta_w>1-2\epsilon\}\cap\pi^{-1}(\{\rho<T+1\})\setminus A$, and a number $\delta>0$ such that $r_s$ is strictly convex in $z$ near $\overline\Omega$ for $|s-s_*|<2\delta$. Finitely many of the sets $(s_*-\delta,s_*+\delta)\times\{|z|<1\}$ cover $Z$; let $\delta_0$ be the smallest of the corresponding $\delta$. Cover $[0,1]$ by closed intervals $I_1,\dots,I_p$ of length $<\delta_0$. For each $a$, let $\Omega_{a,1},\dots,\Omega_{a,K_a}$ be the charts of this cover whose interval meets $I_a$. Then $r_s$ is strictly convex near each $\overline{\Omega_{a,k}}$ for $s\in I_a$, and the balls $\{|z|<1\}$ of these charts cover $M_a=\{x:(s,x)\in Z\text{ for some }s\in I_a\}$. Choose smooth functions $\chi_{a,k}\ge0$ supported in these balls such that $\sum_k\chi_{a,k}\le1$, with equality near $M_a$.

\emph{Bumps.} Since there are only finitely many functions $\chi_{a,k}$ and $r_{s'}\to r_s$ in $C^2$ on compact sets as $s'\to s$, there exists a subdivision $0=s_0<s_1<\dots<s_L=1$ such that each $[s_l,s_{l+1}]$ lies in some $I_{a(l)}$ and the functions
\[
r_{l,k}=r_{s_l}+\sum_{k'\le k}\chi_{a(l),k'}\bigl(r_{s_{l+1}}-r_{s_l}\bigr),\qquad0\le k\le K_{a(l)},
\]
are so close to $r_{s_l}$ in $C^2$ on $\bigcup_k\overline{\Omega_{a(l),k}}$ that they are strictly convex on each $\Omega_{a(l),k}$. Outside these charts, $r_{l,k}=r_{s_l}$; hence $r_{l,k}$ is psh on $E$. At each point, $r_{l,k}$ is a convex combination of $r_{s_l}$ and $r_{s_{l+1}}$, so $r_1\le r_{l,k}\le r_0$, $r_{l,k}=-1$ on $A$, and \eqref{eq:ends} holds.
The difference $r_{l,k-1}-r_{l,k}\ge0$ is supported in the ball $\{|z|<1\}$ of $\Omega_{a(l),k}$. Therefore the pairs $(r_U,r_D)=(r_{l,k-1},r_{l,k})$, $1\le k\le K_{a(l)}$, with the charts $\Omega_{a(l),k}$, ordered lexicographically in $(l,k)$, are convex bumps.

\emph{Matching.} Consecutive bumps with the same $l$ share a function, and $r_{0,0}=r_0$. It remains to compare $r_{l,K}$, where $K=K_{a(l)}$, with $r_{l+1,0}=r_{s_{l+1}}$. They agree near $M_{a(l)}$ and over $\{\rho>T\}$. At any other point $x$, we have $|r_s(x)|>\epsilon$ for all $s\in[s_l,s_{l+1}]$, so $r_{s_l}(x)$ and $r_{s_{l+1}}(x)$ have the same sign and absolute values $>\epsilon$, and $r_{l,K}(x)$ lies between them. Hence $\{r_{l,K}<\lambda\}=\{r_{s_{l+1}}<\lambda\}$ for $|\lambda|\le\epsilon/2$.
\end{proof}

\begin{lemma}[Splitting over a convex bump]\label{lem:radialsplit}
Let $(r_U,r_D)$ be a convex bump, let $\frac14\ge\lambda_+>\lambda_->0$, and let $B_0,B_\infty\subset\C^k$ be balls centred at $0$ with radii $\rho_0>\rho_\infty$. Then there exist $\epsilon_*,C'>0$ such that, for every holomorphic map $\gamma$ on $C_{\lambda_+}\times B_0$ with $\sup|\gamma-\zeta|\le\epsilon_*$, there exist holomorphic maps $\alpha\colon U_{\lambda_-}\times B_\infty\to B_0$ and $\beta\colon V_{\lambda_-}\times B_\infty\to\C^k$ such that $\gamma(x,\alpha(x,\zeta))=\beta(x,\zeta)$ on $C_{\lambda_-}\times B_\infty$, $\sup|\alpha-\zeta|+\sup|\beta-\zeta|\le C'\sup|\gamma-\zeta|$, and $\alpha(\cdot,\zeta)-\zeta$ vanishes to order $r$ along $A$.
\end{lemma}

\begin{proof}
Let $\lambda_j=\lambda_-+(\lambda_+-\lambda_-)2^{-j}$. We verify the hypothesis of Lemma~\ref{lem:nonlin} for the sets $U_{\lambda_j},V_{\lambda_j}$ with $p=n_E$. Applying Lemma~\ref{lem:linsplit} on $D_{\lambda_j}$ with $D_A=U_{\lambda_j}$, $D_B=V_{\lambda_j}$ and the cut-off function $\chi_\Omega$ (Lemma~\ref{lem:pairs}), we split $c$ with $L^2$ norms $\le C_0\sup|c|$, where $C_0=2\mu_m(V_{\lambda_+})^{1/2}+2\|\dbar\chi_\Omega\|_E$ does not depend on $j$ since $V_{\lambda_j}\subset V_{\lambda_+}$ and $\overline{V_{\lambda_+}}$ is a compact subset of $\Omega\subset E\setminus A$. On $V_{\lambda_{j+1}}$ and on $U_{\lambda_{j+1}}\cap\pi^{-1}(\{\rho\le T+1\})$, the mean value inequality on balls of radius $c_1(\lambda_j-\lambda_{j+1})$ (Lemma~\ref{lem:pairs}(d)) gives sup bounds of the form $O(2^{n_Ej})\sup|c|$. At a point $\xi\in U_{\lambda_{j+1}}$ with $\rho(\pi(\xi))\ge T+1$, we have $U_{\lambda_j}=W_{R_j}$ near $\xi$, where $R_j=(1+\lambda_j)^{1/2}e^{-\chi(\rho\circ\pi)/2}$, as well as $\xi\in W_{(1-\delta_j)R_j}$ with $\delta_j\ge(\lambda_j-\lambda_{j+1})/4$ and $q(\xi)\le1$. Lemma~\ref{lem:chi}(b) gives $\Lambda$ with $B(\xi,2/\Lambda)\subset\pi^{-1}(\{\rho\ge T\})$, and Lemma~\ref{lem:pointwise}, with $m-n_E\ge1$, then gives $|a_c(\xi)|\le C_0\delta_j^{-n_E}\sup|c|$.
\end{proof}

\subsubsection*{Gluing}
Let $Y$ be an Oka manifold with a complete Riemannian metric, and let $P_s\subset\C^M$ denote the ball of radius $s$. We use the basic Oka property with approximation for maps from Stein manifolds to $Y$ \cite[Corollary~5.4.5]{Forstneric2017book}, the polynomial convexity of compact convex sets \cite{Hormander1990}, the existence of dominating sprays over holomorphic maps from Stein manifolds \cite[Lemma~5.10.4]{Forstneric2017book}, and the homotopy extension property of $(E,\{r\le\lambda\})$ for regular values $\lambda$. As in \cite[Sects.~5.9--5.10]{Forstneric2017book} and \cite[Sects.~6 and~8]{Stopar2013IJM}, we first glue two sprays over a pair $(U_\lambda,V_\lambda)$ and then extend across a convex bump.

\begin{lemma}[Transition maps; {\cite[Lemma~5.9.3]{Forstneric2017book}}, {\cite[Lemma~6.2]{Stopar2013IJM}}]\label{lem:transition}
Let $S$ be a Stein manifold, let $G\colon S\times P_1\to Y$ be a holomorphic map such that $\partial_\zeta G$ is surjective everywhere, let $L\subset S$ be compact, and let $L_1$ be a compact neighbourhood of $L$. Then, for every $\epsilon_1\in(0,\frac1{16}]$, there exist a neighbourhood $W$ of $L$ and $\delta_1>0$ with the following property: for every $k$ and every holomorphic map $T$ defined near $L_1\times\overline{P_{3/4}}\times\B^k$ with $\dist(T(x,\zeta,w),G(x,\zeta))<\delta_1$ there, there exists a holomorphic map $\gamma_\zeta\colon W\times P_{1/4}\times\B^k\to P_{1/2}$ such that $|\gamma_\zeta-\zeta|<\epsilon_1$ and $T(x,\gamma_\zeta(x,\zeta,w),w)=G(x,\zeta)$.
\end{lemma}

\begin{proof}
We follow \cite[p.~240]{Forstneric2017book}, keeping track of the parameter domain.
Let $\Pi$ be a holomorphic right inverse of $\partial_\zeta G$, which exists by Theorem~B. By the inverse function theorem, $v\mapsto G(x,\zeta+\Pi v)$ has a local inverse $y\mapsto v(x,\zeta,y)$ near $G(x,\zeta)$, uniformly for $(x,\zeta)$ near $L\times\overline{P_{1/2}}$. For sufficiently small $\delta_1$, independent of $k$ and $T$, the map $\eta=\zeta+\Pi v(x,\zeta,T)$ satisfies $T=G(x,\eta)$ and $|\eta-\zeta|<\epsilon_1$; solving $\eta(x,\gamma_\zeta,w)=\zeta'$ by the contraction principle on $\overline B(\zeta',2\epsilon_1)$ (using the Cauchy estimates) gives $\gamma_\zeta$.
\end{proof}

\begin{lemma}[Gluing lemma; cf.\ {\cite[Proposition~5.9.2]{Forstneric2017book}}, {\cite[Proposition~6.1]{Stopar2013IJM}}]\label{lem:gluesprays}
Let $(r_U,r_D)$ be a convex bump, let $\frac14\ge\lambda_1>\lambda_+>\lambda_->0$ and $k\ge0$, and let $H\colon U_{\lambda_1}\times P_1\to Y$ be a product spray that is dominating over $U_{\lambda_1}\setminus A$. Then, for every $\varepsilon'>0$, there exists $\delta'>0$ with the following property.
Let $T$ be a holomorphic map from a neighbourhood of $\overline{V_{\lambda_+}}\times\overline{P_{3/4}}\times\overline{\B^k}$ to $Y$ such that $\partial_{(\zeta,w)}T$ is surjective everywhere and $\dist(T(x,\zeta,w),H(x,\zeta))<\delta'$ on $\overline{C_{\lambda_+}}\times\overline{P_{3/4}}\times\overline{\B^k}$. Then there exist holomorphic maps $\alpha=(\alpha_\zeta,\alpha_w)\colon U_{\lambda_-}\times P'\to P_{1/2}\times\B^k$ and $\beta\colon V_{\lambda_-}\times P'\to P_{1/2}\times\B^k$, where $P'\subset\C^{M+k}$ is the ball of radius $\frac1{16}$, such that $\alpha(\cdot,v)-v$ vanishes to order $r$ along $A$, $|\alpha(x,v)-v|,|\beta(x,v)-v|<\varepsilon'$, and
\[
H'(x,v)=H(x,\alpha_\zeta(x,v))\ \text{on }U_{\lambda_-}\times P',\qquad H'(x,v)=T(x,\beta(x,v))\ \text{on }V_{\lambda_-}\times P'
\]
define a product spray $H'\colon D_{\lambda_-}\times P'\to Y$ that is dominating over $D_{\lambda_-}\setminus A$.
\end{lemma}

\begin{proof}
Choose $\lambda'\in(\lambda_-,\lambda_+)$. Let $\epsilon_*,C'$ be the constants of Lemma~\ref{lem:radialsplit} for the levels $\lambda'>\lambda_-$ and the balls of radii $\frac14,\frac18$ in $\C^{M+k}$, and put $\epsilon_1=\min\{\epsilon_*,1/(32C'),\varepsilon'/(2C')\}$. Applying Lemma~\ref{lem:transition} to $G=H$ on the convex domain $S=C_{\lambda_1}$, which does not meet $A$, with $L=\overline{C_{\lambda'}}$ and $L_1=\overline{C_{\lambda_+}}$, we obtain $W$ and $\delta'=\delta_1$. Then $\gamma(x,v)=(\gamma_\zeta(x,v),w)$, where $v=(\zeta,w)$, is defined on $W\times P_{1/4}\times\B^k$, satisfies $|\gamma-v|<\epsilon_1$, and $T(x,\gamma(x,v))=H(x,\zeta)$. Lemma~\ref{lem:radialsplit} gives $\alpha$ on $U_{\lambda_-}$ and $\beta$ on $V_{\lambda_-}$ such that $\gamma\circ\alpha=\beta$ on $C_{\lambda_-}$, $\sup|\alpha-v|+\sup|\beta-v|\le C'\epsilon_1$, and $\alpha(\cdot,v)-v$ vanishes to order $r$ along $A$. By the Cauchy estimates, $\partial_v\alpha$ and $\partial_v\beta$ are within $\frac12$ of the identity on $P'$. On $C_{\lambda_-}\times P'$, we have $T\circ\beta=T\circ\gamma\circ\alpha=H(x,\alpha_\zeta)$, so $H'$ is well defined; it is dominating by the chain rule, since $H$ is dominating over $U_{\lambda_1}\setminus A$ and $T$ is dominating.
\end{proof}

\begin{lemma}[Extension across a convex bump; cf.\ {\cite[Lemma~5.10.5]{Forstneric2017book}}, {\cite[Proposition~8.3]{Stopar2013IJM}}]\label{lem:homglue}
Let $(r_U,r_D)$ be a convex bump, let $0<\lambda_0<\lambda_1\le\frac14$, where $\lambda_0$ is a regular value of $r_D$, let $F\colon E\to Y$ be a continuous map, and let $H\colon U_{\lambda_1}\times P\to Y$ be a product spray with core $F|_{U_{\lambda_1}}$ that is dominating over $U_{\lambda_1}\setminus A$. Let $K\subset\{r_U<0\}$ be compact and $\varepsilon>0$. Then there exist a continuous map $F'\colon E\to Y$ that agrees with $F$ to order $r$ along $A$ and satisfies $\dist(F',F)<\varepsilon$ on $K$, and a product spray $H'\colon D_{\lambda_0}\times P'\to Y$, $P'\subset\C^{M+\dim Y}$, with core $F'|_{D_{\lambda_0}}$ that is dominating over $D_{\lambda_0}\setminus A$.
\end{lemma}

\begin{proof}
We may assume that $P=P_1$. Put $k=\dim Y$, choose $\lambda_1>\lambda_2>\lambda_3>\lambda_0$, and put $V'=V_{\lambda_1}$. The sets $K_U=\overline{C_{\lambda_2}}\subset K_V=\overline{V_{\lambda_2}}$ are compact and convex in $V'$ (Lemma~\ref{lem:pairs}(a)), and $\overline{V_{\lambda_3}}\subset K_V$.

\emph{Tolerances.} Let $2\delta$ be less than the convexity radius of $Y$ near $F(K_U)$. Choose $\varepsilon'>0$ such that $\dist(H(x,\zeta),F(x))<\min\{\delta/2,\varepsilon\}$ for $x\in K_U\cup K$ and $|\zeta|<\varepsilon'$, let $\delta'$ be the constant given by Lemma~\ref{lem:gluesprays} for $\lambda_1>\lambda_2>\lambda_3$, $k$ and $\varepsilon'$, and put $\delta_1=\min\{\delta',\delta/2\}$.

\emph{Approximation.} Let $\hat H(x,\zeta)=H(x,\psi(x)\zeta)$, where $\psi\in C^\infty_c(U_{\lambda_1}\cap V')$ satisfies $0\le\psi\le1$ and equals $1$ near $K_U$ (so that $\psi(x)\zeta\in P_1$ for $\zeta\in P_1$), and let $\hat H=F$ elsewhere on $V'\times P_1$. The set $K_U\times\overline{P_{7/8}}$ is compact and convex, hence $\cO(V'\times P_1)$-convex. The basic Oka property gives a homotopy $\hat H_t$ from $\hat H$ to a holomorphic map $G$ on $V'\times P_1$ such that $\hat H_t$ is holomorphic near $K_U\times\overline{P_{7/8}}$ and satisfies $\dist(\hat H_t,H)<\delta_1/2$ there; put $g_t=\hat H_t(\cdot,0)$.

\emph{Thickening.} We apply \cite[Lemma~5.10.4]{Forstneric2017book} to the section $(x,\zeta)\mapsto((x,\zeta),G(x,\zeta))$ of $(V'\times P_1)\times Y\to V'\times P_1$ and the compact convex set $K_V\times\overline{P_{3/4}}$. Since this set has contractible Stein neighbourhoods, the lemma gives a spray $\hat G(x,\zeta,w)$, with $w$ near $0\in\C^k$, such that $\hat G(x,\zeta,0)=G(x,\zeta)$ and $\partial_w\hat G$ is surjective near $K_V\times\overline{P_{3/4}}\times\{0\}$. For sufficiently small $\sigma>0$, the map $T(x,\zeta,w)=\hat G(x,\zeta,\sigma w)$ is defined near $K_V\times\overline{P_{3/4}}\times\overline{\B^k}$, $\partial_wT$ is surjective there, and $\dist(T,H)<\delta_1$ on $K_U\times\overline{P_{3/4}}\times\overline{\B^k}$.

\emph{Gluing.} Lemma~\ref{lem:gluesprays} with $\lambda_+=\lambda_2$ and $\lambda_-=\lambda_3$ gives $\alpha$, $\beta$ and a spray $H''$ over $D_{\lambda_3}$ that is dominating off $A$. Its core is $f'=H(x,a(x))$ on $U_{\lambda_3}$, where $a=\alpha_\zeta(\cdot,0)$, and $f'=T(x,b(x))$ on $V_{\lambda_3}$, where $b=\beta(\cdot,0)$; moreover, $|a|,|b|<\varepsilon'$. Since $a$ vanishes to order $r$ along $A$, so does $f'-F=\int_0^1\partial_\zeta H(x,\tau a)a\,d\tau$ in local coordinates, and $\dist(f',F)<\varepsilon$ on $K\subset U_{\lambda_3}$.

\emph{Homotopy.} On $U_{\lambda_3}$, let $\Phi^U_t=H(x,(1-t)a(x))$; on $V_{\lambda_3}$, let $\Phi^V_t=T(x,(1-2t)b(x))$ for $t\le\frac12$ and $\Phi^V_t=g_{2-2t}$ for $t\ge\frac12$, which is continuous since $T(x,0)=G(x,0)=g_1(x)$. On $C_{\lambda_3}\subset K_U$, both homotopies lie within $\delta$ of $F(x)$ by the choice of $\varepsilon'$ and $\delta_1$. Hence $\Phi_t=[\Phi^U_t,\Phi^V_t]_{\chi_\Omega}$, the point dividing the minimal geodesic in the ratio $\chi_\Omega:(1-\chi_\Omega)$, is a homotopy on $\overline{D_{\lambda_0}}$ from $f'$ to $F$ (Lemma~\ref{lem:pairs}(c)). By the homotopy extension property, $\Phi_{1-t}$ extends to a homotopy on $E$ starting at $F$, whose end map $F'$ equals $f'$ on $\overline{D_{\lambda_0}}\supset A\cup K$. Finally, put $H'=H''|_{D_{\lambda_0}\times P'}$.
\end{proof}

\subsubsection*{The approximation theorem}
The sublevel set $\{\theta_w<c\}$ is the disc bundle $\{|\xi|<R(\pi(\xi))\}$, whose radius $R$ is determined by $q(R)=c\,e^{-w}$. The following approximation theorem enlarges the radius from about $e^{-w/2}$ to about $e^{-w'/2}$; it combines Lemmas~\ref{lem:chain} and~\ref{lem:homglue} and corresponds to \cite[Proposition~5.10.1]{Forstneric2017book}.

\begin{proposition}\label{prop:enlargement}
Let $T\ge1$, let $w\ge w'$ be smooth psh functions on $X$ with $w=w'=\chi\circ\rho$ on $\{\rho\ge T\}$, and let $0<\ell\le\frac12$. Let $F\colon E\to Y$ be a continuous map that is holomorphic on $W=\{\theta_w<1+\ell\}$ and admits a product spray with core $F|_W$ that is dominating over $W\setminus A$. Let $K\subset\{\theta_w<\frac12\}$ be compact and $\varepsilon>0$. Then there exist $\ell'\in(0,\frac12]$ and a continuous map $F'\colon E\to Y$ that is holomorphic on $W'=\{\theta_{w'}<1+\ell'\}$, admits a product spray with core $F'|_{W'}$ that is dominating over $W'\setminus A$, agrees with $F$ to order $r$ along $A$, and satisfies $\dist(F',F)<\varepsilon$ on $K$.
\end{proposition}

\begin{proof}
Lemma~\ref{lem:chain} with $\epsilon=\ell/2$ gives convex bumps $(r^{(i)}_U,r^{(i)}_D)$, $i=1,\dots,I$, with charts in $\{\theta_w>1-\ell\}\subset\{\theta_w>\frac12\}$. By Sard's theorem, we can choose $\ell/4\ge\lambda^{(0)}>\lambda^{(1)}>\dots>\lambda^{(I)}>0$ such that $\lambda^{(i)}$ is a regular value of $r^{(i)}_D$ for $i\ge1$. Put $F_0=F$, and restrict the given spray to $\{r_0<\lambda^{(0)}\}=\{\theta_w<1+\lambda^{(0)}\}\subset W$. For $i=1,\dots,I$, we apply Lemma~\ref{lem:homglue} to the $i$-th bump with $\lambda_1=\lambda^{(i-1)}$, $\lambda_0=\lambda^{(i)}$, the set $K$ and $\varepsilon/I$; this is possible since, by the matching of levels in Lemma~\ref{lem:chain}, the previous spray is defined on $\{r^{(i)}_U<\lambda^{(i-1)}\}$, and $K\subset\{r^{(i)}_U<0\}$ since $r^{(i)}_U\le r_0<-\frac12$ on $K$. Finally, $\{r^{(I)}_D<\lambda^{(I)}\}=\{\theta_{w'}<1+\lambda^{(I)}\}$, so $F'=F_I$ and $\ell'=\lambda^{(I)}$ have the required properties.
\end{proof}

To exhaust $E$, we use the following choice of weights.

\begin{lemma}\label{lem:weights}
There exist smooth convex nondecreasing functions $\chi=\varphi_0\ge\varphi_1\ge\cdots$ and numbers $1\le t_1<t_2<\cdots$ such that $\varphi_j=\chi$ on $[t_j,\infty)$ and $L_j\subset\{\theta_{\varphi_j\circ\rho}<\frac12\}$.
\end{lemma}

\begin{proof}
Let $\beta\colon\R\to[0,1]$ be a smooth nondecreasing function with $\beta=0$ on $(-\infty,-1]$ and $\beta=1$ on $[0,\infty)$. For $t\ge1$, put $k_t(\tau)=\chi''(\tau)(1-\beta(\tau-t))$, $h_t(s)=\int_s^t\int_\sigma^tk_t$ (with $h_t=0$ on $[t,\infty)$), and $\varphi^{(t)}=\chi-h_t$. Then $(\varphi^{(t)})''=\chi''\beta(\cdot-t)\ge0$, $(\varphi^{(t)})'=\chi'+\int_s^tk_t\ge0$, $\varphi^{(t)}=\chi$ on $[t,\infty)$, $\varphi^{(t)}$ decreases in $t$, and for $s\le t-2$, $h_t(s)\ge\chi'(t-1)-\chi'(s+1)\to\infty$. We choose $t_j$ inductively so large that $L_j\subset\{\theta_{\varphi^{(t_j)}\circ\rho}<\frac12\}$, and put $\varphi_j=\varphi^{(t_j)}$.
\end{proof}

\begin{proof}[Proof of Theorem~\ref{thm:oka}]
We may assume that $r\ge2$ and $n_E+r-1\ge c+1$. Let $w_j=\varphi_j\circ\rho$ (Lemma~\ref{lem:weights}); then $w_j\ge w_{j+1}$ and $w_j=w_{j+1}=\chi\circ\rho$ on $\{\rho\ge t_{j+1}\}$. By Lemma~\ref{lem:chi}(a), we have $\{\theta_{w_0}<\frac32\}\subset U$; thus $f$ is holomorphic on this set and admits the given spray there.
Applying Proposition~\ref{prop:enlargement} successively to $w_j\ge w_{j+1}$, with $K=L_j$ ($K=\varnothing$ for $j=0$) and $\varepsilon=2^{-j}$, we obtain continuous maps $F_j\colon E\to Y$ with $F_0=f$ that are holomorphic on $\{\theta_{w_j}<1+\ell_j\}\supset L_j$, agree with $f$ to order $r$ along $A$, and satisfy $\dist(F_{j+1},F_j)<2^{-j}$ on $L_j$. Since every compact subset of $E$ lies in some $L_j$, the maps $F_j$ converge locally uniformly to a holomorphic map $\tilde f\colon E\to Y$, which agrees with $f$ to order $r$ along $A$ by the Cauchy estimates.
\end{proof}

\section{Proof of Gromov's conjecture}\label{sec:proofmain}

In this section, we first apply Corollary~\ref{cor:criterion} to
weakly pseudoconvex manifolds admitting a positive line bundle,
and then prove Theorems~\ref{thm-main} and~\ref{thm:local-to-global-spray}.

\begin{theorem}\label{thm:weakly}
Let $X$ be a weakly pseudoconvex complex manifold admitting a
positive line bundle.
Then $X$ admits a local dominating spray.
\end{theorem}

\begin{proof}[Proof of Theorem~\ref{thm:weakly}]
By Takayama's embedding theorem \cite{Takayama1998MA},
there exists an embedding $\iota\colon X\hookrightarrow\PP^N$,
which is not necessarily proper.
Set
$$
H:=\iota^*\mathcal{O}_{\PP^N}(1) \quad \text{and}\quad
B:=K_X\otimes H^{\otimes(2n+1)},
$$
where $n:=\dim X$.
We use the following claim.

\begin{claim}\label{claim-1}
The line bundle $B=K_X\otimes H^{\otimes(2n+1)}$ has the following properties:
\begin{itemize}
\item[$\bullet$] $B$ is positive.
\item[$\bullet$] $B$ is generated by finitely many sections.
\item[$\bullet$] Finitely many sections of $B$ define an embedding of $X$ into a projective space.
\item[$\bullet$] $T_X\otimes B$ is generated by finitely many sections.
\end{itemize}
\end{claim}

\begin{proof}[Proof of Claim~\ref{claim-1}]
Since $\iota\colon X\hookrightarrow\PP^N$ is an embedding,
$H$ is positive and generated by finitely many sections.
The natural surjection $\iota^*\Omega_{\PP^N}^1\to\Omega_X^1$
induces a surjective bundle morphism
$$
\iota^*\bigl(\Omega_{\PP^N}^1(2)\bigr)
\twoheadrightarrow
\Omega_X^1\otimes H^{\otimes2}.
$$
The Euler sequence on $\PP^N$ shows that $\Omega_{\PP^N}^1(2)$
is generated by the sections represented by $z_i\,dz_j-z_j\,dz_i$
for $0\leq i<j\leq N$.
Hence its quotient $\Omega_X^1\otimes H^{\otimes2}$ and all
exterior powers of this quotient are generated by finitely many sections.
This shows that
$$
B\otimes H^\vee\cong\det(\Omega_X^1\otimes H^{\otimes2})
$$
is generated by finitely many sections.
In particular, $B\otimes H^\vee$ admits a metric with semipositive curvature.
Tensoring this metric with the positive metric on $H$,
we conclude that $B$ is positive.
Products of generating sections of $B\otimes H^\vee$ and $H$
generate $B$.
Moreover, let $\psi$ be the morphism defined by finitely many
generating sections of $B\otimes H^\vee$.
The products of these sections with the sections of $H$ defining $\iota$
define the composition of $(\iota,\psi)$ with the Segre embedding.
Since $\iota$ is an embedding, so is this composition.
This proves the third property.
By the canonical isomorphism
$\bigwedge^{n-1}\Omega_X^1\cong T_X\otimes K_X$,
we have
$$
T_X\otimes B
\cong
\bigwedge^{n-1}
\bigl(\Omega_X^1\otimes H^{\otimes2}\bigr)
\otimes H^{\otimes3}.
$$
Both factors on the right are generated by finitely many sections,
which proves the last property.
\end{proof}
Choose finitely many sections
\begin{equation*}
\{\xi_i\}_{i\in I}\subset H^0(X,T_X\otimes B)
\end{equation*}
generating $T_X\otimes B$, and set $m:=|I|$.
Tensoring the corresponding evaluation morphism with $B^\vee$
gives a surjective bundle morphism
$$
\varphi\colon E:=(B^\vee)^{\oplus m}\longrightarrow T_X.
$$
Since $H$ is positive and $X$ is weakly pseudoconvex,
Nakano's vanishing theorem for weakly pseudoconvex manifolds
\cite{Nakano1973,Nakano1974} gives
$$
H^1(X,B)
=H^1\bigl(X,K_X\otimes H^{\otimes(2n+1)}\bigr)=0.
$$
As $S^2E^\vee\otimes E\cong B^{\oplus m\binom{m+1}{2}}$, it follows that
$H^1(X,S^2E^\vee\otimes E)=0$, and hence $\alpha_\varphi(E)=0$.
Corollary~\ref{cor:criterion} therefore yields a local
dominating spray on $E$, completing the proof.
\end{proof}

\begin{proof}[Proof of Theorem~\ref{thm:local-to-global-spray}]
Choose a metric $h_B$ on $B$ with positive curvature, and put $\omega_X=\sqrt{-1}\Theta_{h_B}$, which is a K\"ahler form. Equip $K_X^{-1}\otimes B^{\otimes k}$ with a metric of semipositive curvature, and $K_X^{-1}$ with the induced metric; then $\sqrt{-1}\Theta(K_X^{-1})\ge-k\,\omega_X$. Equip $E=(B^\vee)^{\oplus N}$ with the metric induced by $h_B^{-1}$. Then $\sqrt{-1}\Theta_E=-\omega_X\otimes\id_E$, and
\[
\sqrt{-1}\Theta(K_X^{-1}\otimes\det E)=\sqrt{-1}\Theta(K_X^{-1})-N\omega_X\ge-(k+N)\,\omega_X .
\]
Hence \eqref{eq:hyp} holds with the constant $c=k+N$. We identify $E$ with its total space, let $\pi\colon E\to X$ be the projection and $A$ the zero section, and let $s\colon\Omega\to X$ be the given local dominating spray.

\emph{The map.} Choose a continuous function $\delta>0$ on $X$ such that $\{|\xi|<\delta(\pi(\xi))\}\subset\Omega$, and put $\bar s(\xi)=s(t(\xi)\xi)$, where $t(\xi)=\min\{1,\delta(\pi(\xi))/(2|\xi|)\}$ for $\xi\ne0$ and $t=1$ on $A$. This is a continuous map $E\to X$ that equals $s$ on $U_0=\{|\xi|<\delta(\pi(\xi))/2\}$. The derivative of $s$ along $E_y$ at $0_y$ is $(d^{\mathrm v}s)_y$, which maps onto $T_{X,y}$. Hence Lemma~\ref{lem:fibresubmersive}, applied with $Y=X$ and $g=s|_{U_0}$, gives a neighbourhood $U'\subset U_0$ of $A$ and a product spray with core $\bar s|_{U'}$ that is dominating over $U'\setminus A$.

\emph{Conclusion.} Since $X$ is Oka, Theorem~\ref{thm:oka}, applied with $Y=X$, $f=\bar s$, $U=U'$ and $r=2$, gives a holomorphic map $\tilde s\colon E\to X$ that agrees with $\bar s$ to order $2$ along $A$. Hence $\tilde s(0_y)=s(0_y)=y$, and the differential of $\tilde s$ at $0_y$ coincides with that of $s$. In particular, $d^{\mathrm v}\tilde s=d^{\mathrm v}s$ maps onto $T_X$, so $\tilde s$ is a dominating spray on $E$.
\end{proof}

\begin{proof}[Proof of Theorem~\ref{thm-main}]
Let $H$, $B=K_X\otimes H^{\otimes(2n+1)}$ and $E=(B^\vee)^{\oplus m}$ be as in the proof of Theorem~\ref{thm:weakly}. By Claim~\ref{claim-1}, the line bundle $B$ is positive, and by the proof of Theorem~\ref{thm:weakly}, $E$ admits a local dominating spray.

Since $K_X^{-1}\otimes B\cong H^{\otimes(2n+1)}$ admits a metric with semipositive curvature, induced by the Fubini--Study metric, Theorem~\ref{thm:local-to-global-spray}, applied with $k=1$ and $N=m$, shows that $E=(B^\vee)^{\oplus m}$ admits a dominating spray.
\end{proof}

\begin{corollary}\label{cor:knowncases}
Stein Oka manifolds and projective Oka manifolds are elliptic.
\end{corollary}

\begin{proof}
A Stein manifold is weakly pseudoconvex, and the trivial line bundle with the metric $e^{-\psi}$, for a smooth strictly psh exhaustion function $\psi$, is positive. A projective manifold is compact, hence weakly pseudoconvex, and it has a positive line bundle. A point is elliptic. Hence Theorem~\ref{thm-main} applies.
\end{proof}

Both statements are known: the Stein case is classical \cite[Chap.~5]{Forstneric2017book}, and the projective case is due to Forstneri\v{c} and L\'arusson \cite{Forstneric2026MRL}.

\subsection*{Acknowledgements}
The first author was supported by JSPS KAKENHI Grant-in-Aid
for Early-Career Scientists (JP24K16919).
The second author was partially supported by JSPS KAKENHI
Grant-in-Aid for Scientific Research (B) (JP21H00976)
and the JST FOREST Program (JPMJFR2368).

\subsection*{AI disclosure}
Generative AI tools were used as a supplementary aid for English-language editing and mathematical discussions. The authors reviewed all AI-assisted material, independently verified the mathematical arguments, and take full responsibility for the content of this paper.

{\renewcommand{\addcontentsline}[3]{}

}

\end{document}